\documentclass[10pt,a4paper]{article}

\usepackage[latin1]{inputenc}

\usepackage{amsmath,amscd,amssymb,latexsym,graphicx,color,mathrsfs}

 \usepackage[all,cmtip]{xy}
\usepackage{hyperref}
\hypersetup{
    colorlinks=true,
    linkcolor=black,
    filecolor=black,      
    urlcolor=black,
    pdftitle={Gysin},
    pdfpagemode=FullScreen,
    citecolor=black,
    }

\newtheorem{theorem}{Theorem}[section]

\newtheorem{proposition}[theorem]{Proposition}
\newtheorem{corollary}[theorem]{Corollary}

\newtheorem{lemma}[theorem]{Lemma}
\newtheorem{definition}[theorem]{Definition}
\newtheorem{remark}[theorem]{Remark}

\numberwithin{equation}{section}

\newcommand{\cqfd}{\hfill{\small $\Box$}}
\newenvironment{proof}[1][]{{\bf Proof #1 : }}{\begin{flushright}
\cqfd\end{flushright}}

\reversemarginpar

\newcommand{\theory}{$\mathrm{CS}$}

\newcommand\tgt[1]{{}^{T}\kern-1pt #1}
\newcommand\adi[1]{{}^{ad}\kern-1pt #1}

\newcommand{\dnc}{\mathbb{D}}
\newcommand{\Ind}{\mathrm{ind}}

\title{Gysin maps and wrong way functoriality via geometric deformation groupoids}

\author{P. Carrillo Rouse\footnote{Institut de Math\'ematiques de Toulouse, Universit\'e Toulouse III, 118, Route de Narbonne, 31400
Toulouse, France. Email: paulo.carrillo@math.univ-toulouse.fr This author was partially supported by the ANR OpART, ANR-23-CE40-0016}  and Q. Karegar Baneh Kohal\footnote{ Universidad Nacional Aut\'onoma de M\'exico, Instituto de Matem\'aticas, Unidad Cuernavaca, Cuernavaca, Morelos, M\'exico . Email: quentin.karegar@im.unam.mx}}

\date{\today}

\begin{document}

\maketitle

\begin{center}
{\bf Abstract}
\end{center}
In this article, we study the normal bundle and the deformation to the normal cone functors within the category of Lie groupoids with a view to constructing wrong way pushforward (also called shriek or Gysin) maps in any suitable (compactly supported) cohomology theory for Lie groupoids (including K-theory) by axiomatic means. The main theorems being the functorial behavior of these pushforward maps which recovers, unifies and generalizes many previous cases. In order to illustrate these constructions, we treat the novel case of wrong way maps for groupoid-equivariant (twisted) orbifold K-theory in the last part of the paper.

\section{Introduction and motivation}

The use of deformation groupoids (as the most famous Connes' tangent groupoid) has been a very useful and powerful tool in recent years in index theory, K-theory and related areas. In this paper we study how to use the normal bundle and the deformation to the normal cone functors to get deformation Lie groupoids that allow to geometrically construct pushforward maps in any suitable (co)homology theory for Lie groupoids (not only K-theory), the main results being the functoriality for these pushforward maps which recovers, unifies and generalizes all cases we are aware of and give in this way a geometric approach to them, obtaining as well new cases. The main new example we develop in this paper is the wrong way functoriality for equivariant (twisted) orbifold K-theory as will be detailed below but, as we will explain and develop elsewhere, the constructions and theorems apply for other theories as well.

Now, in order to motivate the results, let us recall the classical construction of a Gysin homomorphism in de Rham cohomology. Consider first the case of an oriented smooth fibration between two smooth manifolds
\begin{equation}
f:M\to B.
\end{equation}
Then the integration of compactly supported forms along the fibers of the fibration defines a group  homomorphism
\begin{equation}
f_!:H_{dR,c}^*(M)\to H_{dR,c}^{*-k}(B)
\end{equation}
between the corresponding compactly supported de Rham cohomology spaces (with $k$ the dimension of the fibers). 
Next, an oriented smooth immersion
\begin{equation}
j:M\hookrightarrow P
\end{equation}
also induces a Gysin homomorphism 
\begin{equation}
j_!:H_{dR,c}^*(M)\to H_{dR,c}^{*+r}(P),
\end{equation}
obtained as follows: first, apply the Thom isomorphism 
$$Th:H_{dR,c}^*(M)\stackrel{\cong}{\longrightarrow}H_{dR,c}^{*+r}(N(j))$$
of the normal bundle $N(j):=j^*TP/TM$,
and then compose it with the canonical extension by zero $$H_{dR,c}^{*+r}(N(j))\to H_{dR,c}^{*+r}(P)$$
after choosing a tubular neighborhood. 
Finally, for an arbitrary oriented smooth map 
\begin{equation}
f:M\to N
\end{equation}
may be decomposed as a composition of an immersion followed by a submersion, reducing the construction of a Gysin homomorphism
\begin{equation}
f_!: H_{dR,c}^*(M)\to H_{dR,c}^{*+r}(N)
\end{equation}
 to the two previous cases.
 By a standard computation, the definition is shown to be  independent of the decomposition, and functorial: 
\begin{equation}
(f\circ g)_!=f_!\circ g_!
\end{equation}
for any composable pair of oriented smooth maps $f$ and $g$.
\par
For the topological $K$-theory, assuming now that $f:M\to N$ is a  $K$-oriented smooth map, the construction of a Gysin map 
\begin{equation}
f_!: K_{top}^*(M)\to K_{top}^{*+r}(N).
\end{equation}
is done along the same lines, by remplacing the integration along the fibers by a family index morphism.
In this latter case, proving that the above map does not depend on the decomposition (immersion-submersion) and the corresponding functorial behavior is not trivial since it corresponds to the Atiyah-Singer index theorem (for families). The comparison between the respective Gysin maps, in topological K-theory and  de Rham cohomology, returns the classical Atiyah-Singer index formula (including the family case).
A drawback of the above strategy of decomposing into an immersion-submersion, is that it cannot be applied in the equivariant case (under the action of a discrete group or groupoid).

In this paper, we propose a geometric construction for Gysin morphisms in ``theories'' including  two above examples. 
In particular, our construction and functoriality theorems apply for general Lie groupoids and will give as applications (listed below) all known cases, that we are aware, of Gysin morphisms and their functoriality. We will as well give at least one (highly not trivial in our opinion) not known (or at least not stablished in the litterature) case, the wrong way functoriality for (twisted) equivariant K-theory for orbifolds. As we will also argue later, the construction is  shaped in order to work for suitable cohomology theory on differentiable stacks. But before entering into the details of the present paper let us give an idea of the contruction for the above two cases, for us it is important to give this motivating example to justify why the construction is geometric and natural.

Consider a smooth map $f:M\to N$. In the category of smooth manifolds $f$ is not necessary an immersion nor a submersion, but viewing it as a map between differentiable stacks it can be easily modified into a Lie groupoid immersion. 
Indeed, it fits in the commutative diagram
\begin{equation}
\xymatrix{
M\ar[r]^-f\ar[rd]_-{\tilde{f}}&N\\
&(M\times M)\times N\ar[u]_-p^-\sim
}
\end{equation}
where $\tilde{f}=(\Delta,f)$ with $\Delta:M\to M\times M$ the diagonal map, and $p$ the canonical projection. In order to be interesting, the above diagram should be considered as a commutative diagram of Lie groupoid morphisms, where the only not unit groupoid is the pair groupoid $M\times M\rightrightarrows M$. Once in the category of Lie groupoids and generalized morphisms not only the above diagram is commutative but the projection $p$ is a canonical Morita equivalence, hence in the category of differentiable stacks the map $f$ and the $\tilde{f}$ are indistinguishable. Now, the advantage of working with $\tilde{f}$ is that there is an associated normal bundle that takes the form in this example of the action Lie groupoid
\begin{equation}
TM\ltimes f^*TN\rightrightarrows f^*TN
\end{equation}
induced by the differential $df$ and called the normal groupoid of $f$. Furthermore, associated to such an $\tilde{f}$ there is a Lie groupoid (the deformation - to the normal cone - groupoid) $D\rightrightarrows D^{(0)}$ that comes with an extra structure of a smooth family of Lie groupoids parametrized by the closed interval $[0,1]$, ruled by a smooth map $q:D\to [0,1]$, such that the fiber at zero, $q^{-1}(0)$, identifies with the normal groupoid $TM\ltimes f^*TN$ and the fiber at $t\neq 0$, $q^{-1}(t)$, identifies with the Lie groupoid $(M\times M)\times N$. 

So, suppose there is a theory $F^\bullet$ (for example $F^\bullet=K^*_{top}$ or $F^\bullet=H_{dR,c}^*$ extended to Lie groupoids) attaching a module to each Lie groupoids, and which satisfies some suitable axioms like  homotopy invariance, Morita invariance and  functoriality for saturated inclusion. In particular, the restriction at zero of the deformation groupoid
\begin{equation}
r_0:F^\bullet(D)\to F^\bullet(TM\ltimes f^*TN)
\end{equation}
is an isomorphism. If moreover, an analog of the Thom isomorphism holds, then one obtains a pushforward map 
\begin{equation}
f_!:F^\bullet(M)\to F^\bullet(N)
\end{equation}
by composing
\begin{enumerate}
\item the Thom isomorphism
\begin{equation}
Th: F^\bullet(M)\to F^\bullet(TM\ltimes f^*TN)
\end{equation}
followed by
\item a deformation index
\begin{equation}
\Ind_f:F^\bullet(TM\ltimes f^*TN)\to F^\bullet((M\times M)\times N)
\end{equation}
(defined by $r_1\circ r_0^{-1}$, the composition of the inverse of the restriction at zero induced morphism followed by the restriction at one induced morphism)
and followed by
\item an isomorphism 
\begin{equation}
F^\bullet((M\times M)\times N)\to F^\bullet(N)
\end{equation}
induced by the canonical Morita equivalence.
\end{enumerate}

For $F^\bullet=K^*_{top}$ or $F^\bullet=H^*_{dR,c}$, it is not hard to show by direct computations that the above defined Gysin morphism coincides with the "classically" defined ones. Of course one has to extend these theories to at least the groupoids used in the construction, for $K^*_{top}$ one can use the $K$-theory groups for the associated $C^*$-algebras, for $H^*_{dR,c}$ one can use for instance the cyclic periodic homology of appropriate Schwartz type algebras, as done for example in \cite{CRWW}, we will come to more on this below.

Let us now describe in more details the contents of this paper. Given a commutative ring $R$, assume that $F^\bullet$ associate to each Lie groupoid $G$ some (graded) $R$-module  $F^\bullet(G)$, and satisfies the following axioms:
\begin{enumerate}
\item[(A1)] if $G_1 \subseteq G$ is a closed saturated inclusion of Lie groupoids, then there is a restriction morphism
$$r_{G_1}^{G}:F^\bullet(G)\to F^\bullet(G_1)$$
such that $r_{G}^{G}=Id$. 
In addition, if $G_2 \subseteq G_1 \subseteq G$ are closed saturated inclusion, then
$$r_{G_2}^{G}=r_{G_2}^{G_1}\circ r_{G_1}^{G}.$$
\item[(A2)] For any Lie groupoid immersion $(G_2,G_1,j)$ ($j:G_1\to G_2$ a not necessarilly injective immersion of Lie groupoids), if $\dnc(G_2,G_1,j)$ denotes the associated deformation to the normal groupoid (see section \ref{sectionDCN+normal} for precise definition), then the restriction morphism
$$r_0 := r_{N(G_2,G_1,j)\times \{0\}}^{\dnc(G_2,G_1,j)}: \: F^\bullet(\dnc(G_2,G_1,j))\to F^\bullet(N(G_2,G_1,j)),$$
where $N(G_2,G_1,j)$ is the associated normal bundle groupoid (see section \ref{sectionDCN+normal} for precise definition), is an isomorphism.

\end{enumerate}
We call such an $F^\bullet$, a \theory-theory (Cohomological Stack). An example of this is $K^*_{top}(C^*(\cdot))$ (the $K$-theory groups of the associated maximal $C^*$algebras (=reduced in some good cases)). 

Let $F^\bullet$ be a \theory-theory and let $(G_2,G_1,j)$ be a Lie groupoid immersion, we can define {\bf the $F^\bullet$-deformation index of the immersion $(G_2,G_1,j)$} (or the deformation index in the theory $F^\bullet$) as the morphism of $R$-modules
\begin{equation}
\Ind_{F^\bullet}^{j}: F^\bullet(N(G_2,G_1,j))\to F^\bullet(G_2)
\end{equation}
given as
$$\Ind_{F^\bullet}^{j}:=r_1\circ r_0^{-1},$$
where $r_1:=r_{G_2\times \{1\}}^{\dnc(G_2,G_1,j)}$.

Our first fuctoriality theorem is an immediate consequence of the above required properties for an \theory-theory and even more importantly of the geometric functoriality properties of the normal bundle functor and of the deformation to the normal cone functor. Indeed, its proof becomes almost trivial once one understand these geometric functorial properties: the symmetry theorems~\ref{symmetryNBthm} and~\ref{symmetryDNC}, which describe how the iteration of the normal bundle functor and of the deformation to the normal cone functor are related. 

\begin{theorem}[Pushforward functoriality I]
Let $F^\bullet$ be a \theory-theory as above. Suppose we have a commutative diagram of Lie groupoid immersions
\begin{equation}
\xymatrix{
G_1\ar[r]^-{i_1}\ar[d]_-{j_1}&G_2\ar[d]^-{j_2}\\
H_1\ar[r]_-{i_2}&H_2.
}
\end{equation}
Suppose the normal bundle maps $N(j_2,j_1)$ and $N(i_2,i_1)$ are Lie groupoid immersions as well.
Then the following equality of $R$-modules index morphisms holds
\begin{equation}
\Ind_{F^\bullet}^{i_2}\circ \Ind_{F^\bullet}^{N(j_2,j_1)}=\Ind_{F^\bullet}^{j_2}\circ \Ind_{F^\bullet}^{N(i_2,i_1)}
\end{equation}
\end{theorem}

The case we are interested in is when $G_1=H_1$ and $j_1=Id$, in this case the hypothesis are satisfied, corollary \ref{cor-monosquare-specialcase}. Now, the reason for considering general squares is that, in our opinion, it highly simplifies the double deformation arguments and it reflects better the symmetry properties of the normal bundle functor and of the deformation to the normal cone functor.

Next, in order to go further into the functoriality and the definition of the Gysin maps we want to consider, we can ask to our \theory-theories to have what we call the Thom isomorphism property, developped in section \ref{sectiondefindices}, particularly around the definition \ref{Thomproperties}. There, we consider what we call $F^\bullet$-orientations on vector bundles groupoids which are essentially the existence of what we can call Thom isomorphisms, for more details see definition \ref{Thomproperties}. Of particular importance for us: given $(G_2,G_1,j)$ a Lie groupoid immersion and a given \theory-theory $F^\bullet$, we say that $j$ is $F^\bullet$-oriented if its normal bundle groupoid has an $F^\bullet$-orientation, in this case we have then a Thom isomorphism
\begin{equation}
F^\bullet(G_1)\xrightarrow{Thom_j} F^\bullet(N(G_2,G_1,j)).
\end{equation}
that allows, together with the associated deformation index, to define
$$j_!:F^\bullet(G_1)\to F^\bullet(G_2)$$
as the composition of the Thom isomorphism $Thom_j$ and the deformation index $\Ind_{F^\bullet}^{j}$. 

The following theorem is our main result, as we will see below it will give as immediate applications very interesting Gysin maps and functoriality theorems, for some general maps, not necessarilly for immersions. Its proof is quite straightforward from the theorem above together with the axioms/properties asked to the Thom isomorphisms, properties that are of course classical Thom properties for (generalized) cohomological theories such as K-theory or other cohomologies. Our way of presenting the Thom isomorphism property is largely inspired from the work of Hilsum and Skandalis in \cite{HS} and the work of Emerson and Meyer in \cite{EM2010}.

\begin{theorem}[Pushforward functoriality II]
Let $F^\bullet$ be a \theory-theory with the Thom isomorphism property. Suppose we have a commutative diagram of Lie groupoid immersions
\begin{equation}
\xymatrix{
G_1\ar[r]^-{i_1}\ar[d]_-{j_1}&G_2\ar[d]^-{j_2}\\
H_1\ar[r]_-{i_2}&H_2
}
\end{equation}
for which $N(j_2,j_1)$ and $N(i_2,i_1)$ are Lie groupoid immersions. Suppose further that the Lie groupoid immersions $\dnc(j_2,j_1)$ and $\dnc(i_2,i_1)$ are $F$-oriented.
Then the following equality of $R$-modules index morphisms holds
\begin{equation}\label{pushforwardfunctIequation}
 j_{2!}\circ i_{1!}= i_{2!}\circ j_{1!}
\end{equation}
\end{theorem}

\subsection*{Applications, examples}

The following theorem is based on the idea exposed above in which we managed to replace a given map into a Morita equivalent one which is an immersion. This idea allows to get the following general theorem as a corollary of the above in a very explicit way. Now, the only extra axioms/properties we need to ask to our \theory-theories are Morita invariance and a geomeric property we call "orientation deformation property", these two properties are again satisfied by the example above of $K$-theory but also for any cohomological theory for differentiable stacks. See definitions \ref{defMoritainv} and \ref{defOdefproperty} for more details on these two properties. The theorem is the following:

\begin{theorem}[Groupoid equivariant wrong way $F^\bullet$-functoriality for \'etale Lie groupoids]\label{Introthmetale}
Consider a \theory-theory $F^\bullet$ with the Thom isomorphism property, satisfying Morita invariance and with the orientation deformation property. Let $G\rightrightarrows M$ be a Lie groupoid.
Let 
\begin{equation}
\xymatrix{
G_1\ar[rd]_-{f_3}\ar[r]^-{f_1}&G_2\ar[d]^-{f_2}\\
&G_3
}
\end{equation}
be a commutative diagram of  $F^\bullet$-oriented $G$-equivariant groupoid morphisms between \'etale Lie groupoids. Then, the following equality of $R$-module morphisms holds:
\begin{equation}
f_{2!}\circ f_{1!}=f_{3!}: F^\bullet(G_1\rtimes G)\longrightarrow F^\bullet(G_3\rtimes G)
\end{equation}
\end{theorem}

The above theorem applies directly in, at least, the following three explicit cases: 

\begin{enumerate}
\item The twisted K-theory wrong way functoriality for $G$-manifolds (for a Lie groupoid $G$). Indeed, the main result in \cite{CaWangAENS}, theorem 4.2 in loc.cit, is an immediate consequence of the above theorem. In fact the proof of that theorem was exactly the model proof whose essence we succeded to abstract in the main theorem above. The theorem 4.2 in loc.cit, in turn,  implies Atiyah-Singer's families index theorem in any groupoid equivariant case whenever a topological index is defined.
\item The wrong way functoriality for delocalised cohomology associated to discrete groups acting properly on smooth manifolds. Such functorial property is one of the main theorems in \cite{CRWW}, which was tackled there using the ideas from \cite{CaWangAENS}. Now, the theorem above provides a complete and independent proof. One of the main results in \cite{CRWW} gives a natural relation, via Chern characters, between the wrong way pushforward maps in K-theory and the ones in delocalised cohomology fitting in a Riemann-Roch diagram that implies Atiyah-Singer's index formulae for families in the equivariant setting (for a discrete group acting). 
\item The longitudinal index theorem of Connes-Skandalis, as well as its twisted version. This was already explained in \cite{CaWangI} using as well deformation groupoid techniques.
\end{enumerate}

Let us now state the main new explicit result obtained from the previous theorems.
The result is developped in section \ref{subsectionOrbifold} where we verify that the hypothesis of theorem \ref{Introthmetale} are satisfied once reduced to the \'etale case. In this paper, an orbifold groupoid stands for a proper \'etale Lie groupoid. 

\begin{theorem}[Wrong way functoriality for groupoid equivariant K-theory for orbifold groupoids]
Let $G$ be a Lie groupoid.
Let 
\begin{equation}
\xymatrix{
\mathfrak{X}_ 1\ar[r]^-{f}\ar[dr]_-{h}&\mathfrak{X}_2\ar[d]^-{g}\\
&\mathfrak{X}_3
}
\end{equation}
be a commutative diagram $K$-oriented $G$-equivariant Lie groupoid morphisms between orbifold groupoids. Then, the following diagram of abelian group morphisms commutes:
\begin{equation}
\xymatrix{
K^*_G(\mathfrak{X}_ 1)\ar[r]^-{f_!}\ar[dr]_-{h_!}&K^*_G(\mathfrak{X}_2)\ar[d]^-{g_!}\\
&K^*_G(\mathfrak{X}_3)
}
\end{equation}
\end{theorem}

\begin{remark}
\begin{itemize}
\item One of the main difficulty in order to prove the latter result is to show the existence of the Thom isomorphism.
Such isomorphism is well-known for the K-theory of vector bundles over spaces (including the groupoid equivariant case),  nevertheless in the case of an arbitrary vb-groupoids it is, as far as we are aware, still  unknown.
The étaleness assumption plays a crucial role here, since it allows an explicit description of the normal vb-groupoid, and consequently implies the existence of a Thom isomorphism. This is explained at the beginning of section \ref{subsectionOrbifold}.
\item In the setting above, the condition of being K-oriented is not restrictive if one assumes that $G$ acts properly in the domain of the corresponding map. This is how it will be used in a future work to assembly these pushforward maps as was done for the manifold case in \cite{CRWW}.
\item The twisted version of the theorem above, which give the exact generalization to orbifolds of theorem 4.2 in \cite{CaWangAENS}, is discussed in section \ref{subsectiontwistedorb}.
\item We have been saying that we can apply similar ideas to other theories, and it is indeed the case, for example as mentioned above to delocalised cohomology of discrete groups acting on smooth manifolds. In order to pass from smooth manifolds to orbifolds or more sophisticated geometric objects, one needs to properly introduce and work with appropriate Schwartz type algebras on the deformation orbifold groupoids. We will see how this is possible and work the details in a future work.
\item The properties we are asking to our \theory-theory seem to be quite standard in other groupoid cohomological theories. So even without the use of algebras, there is the possibility to construct wrong way maps and obtain functoriality using deformation groupoids if one has indeed these properties. Definitely a direction to explore.
\end{itemize}
\end{remark}

{\bf Acknowledgements:} We would like to thank Thomas Schick and Georges Skandalis for very stimulating and inspiring exchanges and conversations. 


\section{Deformation and normal groupoids and deformation indices for immersions}\label{sectionDCN+normal}

For basic definitions, properties and notations on Lie groupoids the reader may have a look to section 2 in \cite{CaWangAENS}. Classic references on Lie groupoids are \cite{MM} or \cite{Mac}. 

We will be denoting a Lie groupoid by $G\rightrightarrows G_0$ as classically whenever we need to emphasize the unit manifold $G_0$. 
%
%
%
\subsection{Normal groupoid associated to an immersion}
Let $(X,Y,j)$ be a  triple consisting of a couple of  $C^\infty$-manifolds $X$ and $Y$ together with a (not necessarily injective) immersion $j:Y\to X$. For such a triple we can define the normal bundle as the quotient vector bundle over $Y$
\begin{equation}
N(X,Y,j)=j^*TX/TY.
\end{equation}

A morphism between triples $(X,Y,j)$ and $(X',Y',j')$ as above consist of a couple $(f,f_0)$ of $C^\infty$ maps such that the following diagram is commutative
\begin{equation}
\xymatrix{
X\ar[r]^-{f}&X'\\
Y\ar[u]^-j\ar[r]_-{f_0}&Y'\ar[u]_-{j'}.
}
\end{equation}
Such commutative square induces a vector bundle map, named normal differential of $(f,f_0)$, 
\begin{equation}
N(f,f_0):N(X,Y,j)\to N(X',Y',j')
\end{equation}
induced in a canonical way by the differential $df:TX \to TX'$.
The following lemma is an immediate consequence of the chain rule. 
\begin{lemma}[Functoriality of the normal bundle construction]
Given two morphisms of triples as above
\begin{equation}
\xymatrix{
X\ar[r]^-{f}&X'\ar[r]^-{g}&X''\\
Y\ar[u]^-j\ar[r]_-{f_0}&Y'\ar[u]_-{j'}\ar[r]_-{g_0}&Y''\ar[u]_-{j''}.
}
\end{equation}
we have that
\begin{equation}
N(g\circ f,g_0\circ f_0)=N(g,g_0)\circ N(f,f_0).
\end{equation}
Moreover
\begin{equation}
N(Id_X,Id_Y)=Id_{(X,Y,j)}.
\end{equation}
\end{lemma}

\begin{corollary}
Let $(G_2,G_1,j)$ be a Lie groupoid immersion. Then,  there is a Lie groupoid structure on 
\begin{equation}
N(G_2,G_1,j)\rightrightarrows N(G_2^{(0)},G_1^{(0)},j_0)
\end{equation}
induced by the normal bundle construction. That is, with Lie groupoid structural maps given by 
$s=N(s_2,s_1), t=N(t_2,t_1), m=N(m_2,m_1), u=N(u_2,u_1), \iota=N(\iota_2,\iota_1)$.
This Lie groupoid is called the normal (bundle) groupoid associated to $(G_2,G_1,j)$.
\end{corollary}

{\bf Notation:} Given a Lie groupoid immersion as above, we will use the short handed notation

$$N(j)\rightrightarrows N(j_0)$$
for the associated normal groupoid.

\subsubsection{Symmetry of the normal functor and of the normal groupoids}

One of the fondamental properties that we will need for the pushforward functoriality will be what we call the symmetry property of the normal functor that we now explain:

Suppose there is a commutative diagram of manifold immersions
\begin{equation}\label{diag4grpds}
\xymatrix{
L_1\ar[r]^-{i_1}\ar[d]_-{j_1}&L_2\ar[d]^-{j_2}\\
M_1\ar[r]_-{i_2}& M_2,
}
\end{equation}
such that the normal bundle maps $N(j_2,j_1)$ and $N(i_2,i_1)$ are immersions as well. In the case the immersions are embeddings or inclusions of submanifolds it is not hard to see that there is a diffeomorphism between the total spaces
$$N(N(j_2,j_1))\cong N(N(i_2,i_1)).$$
In general, this property still holds. In fact, by combining straightforward computations like the one described above together with the functoriality of the normal bundle construction we have the following property. This is the main content of the article \cite{Kar}.

\begin{theorem}[Normal groupoid symmetry]\label{symmetryNBthm}
Consider a commutative square of Lie groupoid immersions
\begin{equation}\label{diag4grpds}
\xymatrix{
G_1\ar[r]^-{i_1}\ar[d]_-{j_1}&G_2\ar[d]^-{j_2}\\
H_1\ar[r]_-{i_2}&H_2,
}
\end{equation}
 such that the normal bundle maps $N(j_2,j_1)$ and $N(i_2,i_1)$ are Lie groupoids immersions as well (between the respective normal groupoids). Then, there is a canonical Lie groupoid isomorphism 
\begin{equation}
\xymatrix{
N(N(j_2,j_1))\ar@<.5ex>[d]\ar@<-.5ex>[d] \ar[r]^-\cong & N(N(i_2,i_1)) \ar@<.5ex>[d]\ar@<-.5ex>[d] \\
N(N(j_2^{(0)},j_1^{(0)}))\ar[r]_-\cong & N(N(i_2^{(0)},i_1^{(0)})).
}
\end{equation}
between the normal groupoids associated to $N(j_2,j_1)$ and $N(i_2,i_1)$.

\end{theorem}

\subsection{Deformation to the normal cone groupoid associated to an immersion}

Let $(X,Y,j)$ be a  triple as in the previous section. For such a triple we can consider the set
\begin{equation}
\dnc(X,Y,j)=N(X,Y,j)\bigsqcup X\times \mathbb{R}^*.
\end{equation}

We will briefly describe the deformation to the normal cone $C^\infty$-structure on the set above. Let us first consider the case where $M=\mathbb{R}^p\times \mathbb{R}^{n-p}$
and $X=\mathbb{R}^p \times \{ 0\}$ ( here we
identify  $X$ canonically with $ \mathbb{R}^p$). We denote by
$q=n-p$ and by $\dnc_{p}^{n}$ for $(\mathbb{R}^n,\mathbb{R}^p)$ as above. In this case
we   have that $\dnc_{p}^{n}=\mathbb{R}^p \times \mathbb{R}^q \times \mathbb{R}$ (as a
set). Consider the
bijection  $\psi: \mathbb{R}^p \times \mathbb{R}^q \times \mathbb{R} \rightarrow
\dnc_{p}^{n}$ given by
\begin{equation}\label{psi}
\psi(x,\xi ,t) = \left\{
\begin{array}{cc}
(x,\xi ,0) &\mbox{ if } t=0 \\
(x,t\xi ,t) &\mbox{ if } t\neq0
\end{array}\right.
\end{equation}
whose  inverse is given explicitly by
$$
\psi^{-1}(x,\xi ,t) = \left\{
\begin{array}{cc}
(x,\xi ,0) &\mbox{ if } t=0 \\
(x,\frac{1}{t}\xi ,t) &\mbox{ if } t\neq0
\end{array}\right.
$$
We can consider the $C^\infty$-structure on $\dnc_{p}^{n}$
induced by this bijection.

We pass now to the general case. A local chart
$(\mathscr{U},\phi)$ of $M$ at $x$  is said to be an $X$-slice   if
\begin{itemize}
\item[1)]  $\mathscr{U}$  is an open neighborhood of $x$ in $M$ and  $\phi : \mathscr{U}  \rightarrow U \subset \mathbb{R}^p\times \mathbb{R}^q$ is a diffeomorphsim such that $\phi(x) =(0, 0)$.
\item[2)]  Setting $V =U \cap (\mathbb{R}^p \times \{ 0\})$, then
$\phi^{-1}(V) =   \mathscr{U} \cap X$ , denoted by $\mathscr{V}$.
\end{itemize}
With these notations understood, we have $(U,V)\subset \dnc_{p}^{n}$ as an
open subset.   For $x\in \mathscr{V}$ we have $\phi (x)\in \mathbb{R}^p
\times \{0\}$. If we write
$\phi(x)=(\phi_1(x),0)$, then
$$ \phi_1 :\mathscr{V} \rightarrow V \subset \mathbb{R}^p$$
is a diffeomorphism.  Define a function
\begin{equation}\label{phi}
\tilde{\phi}:(\mathscr{U},\mathscr{V}) \rightarrow (U,V)
\end{equation}
by setting
$\tilde{\phi}(v,\xi ,0)= (\phi_1 (v),d_N\phi_v (\xi ),0)$ and
$\tilde{\phi}(u,t)= (\phi (u),t)$
for $t\neq 0$. Here
$d_N\phi_v: N_v \rightarrow \mathbb{R}^q$ is the normal component of the
 derivative $d_v\phi$ for $v\in \mathscr{V}$. It is clear that $\tilde{\phi}$ is
 also a  bijection. In particular,  it induces a $C^{\infty}$ structure on $_{\mathscr{V}}^{\mathscr{U}}$.
Now, let us consider an atlas
$ \{ (\mathscr{U}_{\alpha},\phi_{\alpha}) \}_{\alpha \in \Delta}$ of $M$
 consisting of $X$-slices. The following results are proved in \cite{Ca2}

\begin{proposition}[proposition 3.1, \cite{Ca2}]
The charts recalled above form a $C^\infty$-atlas on $\dnc(X,Y,j)$. We refer as the deformation to the normal cone to the $C^\infty$-manifold $\dnc(X,Y,j)$ with $C^\infty$-structure defined by this atlas. \\
The deformation to the normal cone $\dnc(X,Y,j)$ is Hausdorff if and only if $j$ is an injective immersion.
\end{proposition}

\begin{proposition}[proposition 3.4, \cite{Ca2}]
Let $(f,f_0)$ be a morphism between triples $(X,Y,j)$ and $(X',Y',j')$ as above. The morphism
\begin{equation}
\dnc(f,f_0):\dnc(X,Y,j)\to \dnc(X',Y',j')
\end{equation}
defined by 
\begin{equation}
\dnc(f,f_0)(V,0)=(N(f,f_0)(V),0)
\end{equation}
for $V\in N(X,Y,j)$, and by
\begin{equation}
\dnc(f,f_0)(x,t)=(f(x),t)
\end{equation}
for $(x,t)\in X\times \mathbb{R}^*$, is a $C^\infty$-map.
\end{proposition}

The following lemma is a straightforward computation

\begin{lemma}[Functoriality properties of the DNC construction]
Given two morphisms of triples as above
\begin{equation}
\xymatrix{
X\ar[r]^-{f}&X'\ar[r]^-{g}&X''\\
Y\ar[u]^-j\ar[r]_-{f_0}&Y'\ar[u]_-{j'}\ar[r]_-{g_0}&Y''\ar[u]_-{j''}.
}
\end{equation}
we have that
\begin{equation}
\dnc(g\circ f,g_0\circ f_0)=\dnc(g,g_0)\circ \dnc(f,f_0).
\end{equation}
Moreover
\begin{equation}
\dnc(Id_X,Id_Y)=Id_{\dnc(X,Y,j)}.
\end{equation}
\end{lemma}

Using, as for the normal bundle construction, the functoriality properties above, we get the following corollary.

\begin{corollary}
Let $(G_2,G_1,j)$ be a Lie groupoid immersion. Then there is a Lie groupoid structure on 
\begin{equation}
\dnc(G_2,G_1,j)\rightrightarrows \dnc(G_2^{(0)},G_1^{(0)},j_0)
\end{equation}
induced by the normal bundle construction. That is, with Lie groupoid structural maps given by 
$s=\dnc(s_2,s_1), t=\dnc(t_2,t_1), m=\dnc(m_2,m_1), u=\dnc(u_2,u_1), \iota=\dnc(\iota_2,\iota_1)$.
This Lie groupoid is called the deformation (to the normal cone) groupoid associated to $(G_2,G_1,j)$.
\end{corollary}

Now, the symmetry property \ref{symmetryNBthm} satisfied by the normal bundle construction has a very important consequence for us. Indeed, suppose there is a commutative diagram of Lie groupoid immersions
\begin{equation}\label{diag4grpds}
\xymatrix{
G_1\ar[r]^-{i_1}\ar[d]_-{j_1}&G_2\ar[d]^-{j_2}\\
H_1\ar[r]_-{i_2}&H_2,
}
\end{equation}
such that the normal bundle maps $N(j_2,j_1)$ and $N(i_2,i_1)$ are immersions as well. Then there is an associated deformation Lie groupoid morphism
\begin{equation}
\dnc(G_2,G_1,i_1)\xrightarrow{\dnc(j_2,j_1)}\dnc(H_2,H_1,i_2)
\end{equation}
which is an immersion. Thus, the associated normal groupoid 
\begin{equation}
\mathfrak{N}:=N(\dnc(j_2,j_1))\rightrightarrows N(\dnc(j_2^{(0)},j_1^{(0)})).
\end{equation}
exists. On the other hand, let us consider the deformation Lie groupoid
\begin{equation}
\dnc(N(i_2,i_1))\rightrightarrows \dnc(N(i_2^{(0)},i_1^{(0)}))
\end{equation}
associated to the normal Lie groupoid $N(i_2,i_1)\rightrightarrows N(i_2^{(0)},i_1^{(0)})$.
The following result is a straightforward corollary of theorem \ref{symmetryNBthm}.

\begin{corollary}\label{symmetryDNC}
Under the notations above, there is a canonical isomorphism of Lie groupoids
\begin{equation}
\xymatrix@R=1.5pc{
N(\dnc(j_2,j_1))\ar@<.5ex>[d]\ar@<-.5ex>[d] \ar[r]^-\cong & \dnc(N(i_2,i_1)) \ar@<.5ex>[d]\ar@<-.5ex>[d] \\
N(\dnc(j_2^{(0)},j_1^{(0)}))\ar[r]_-\cong & \dnc(N(i_2^{(0)},i_1^{(0)})).
}
\end{equation}
\end{corollary}

\subsection{Deformation indices for Lie groupoid immersions and functoriality}\label{sectiondefindices}

In this section we will define the deformation index associated to a Lie groupoid immersion and to a given generalized theory (later on our main examples will be the K-theory of the associated $C^*$-algebras but also the Periodic cyclic theory of the associated LF-algebras).

For the moment we do not want to use the term functor but let us suppose that for a given Lie groupoid $G$ we have an associated  $R$-module $F^\bullet(G)$ ($R$ any ring) that satisfies axioms (A1) and (A2) from the introduction. 
Recall that (A1) is the requirement that $F^\bullet$ is a presheaf from the category of Lie groupoid equipped with closed saturated inclusion as morphisms to the category of graded $R$-module.
The axiom (A2) asserts that the restriction to the special fiber ($t=0$) of the deformation groupoid induces an isomorphism of $R$-modules.
%
We can state our first general functoriality result.

\begin{theorem}[Pushforward functoriality I]\label{pushforwardfunctI}
Let $F^\bullet$ an association as above. Suppose we have a commutative diagram of Lie groupoid immersions
\begin{equation}\label{diag4grpds}
\xymatrix{
G_1\ar[r]^-{i_1}\ar[d]_-{j_1}&G_2\ar[d]^-{j_2}\\
H_1\ar[r]_-{i_2}&H_2.
}
\end{equation}
Suppose the normal bundle maps $N(j_2,j_1)$ and $N(i_2,i_1)$ are immersions as well.
Then the following equality of $R$-modules index morphisms holds
\begin{equation}\label{pushforwardfunctIequation}
\Ind_{F^\bullet}^{i_2}\circ \Ind_{F^\bullet}^{N(j_2,j_1)}=\Ind_{F^\bullet}^{j_2}\circ \Ind_{F^\bullet}^{N(i_2,i_1)}
\end{equation}
\end{theorem}

\begin{proof}
By the functoriality properties of the DNC construction, and because on the hypothesis on 
$N(j_2,j_1)$ and $N(i_2,i_1)$ being immersions, we have a Lie groupoid immersion
\begin{equation}
\dnc(G_2,G_1,i_1)\stackrel{\dnc(j_2,j_1)}{\longrightarrow}\dnc(H_2,H_1,i_2).
\end{equation}
Consider the associated deformation Lie groupoid
\begin{equation}
\mathfrak{D}:=\dnc(\dnc(H_2,H_1,i_2),\dnc(G_2,G_1,i_1),\dnc(j_2,j_1)).
\end{equation}
The following diagram is commutative by the assumptions on $F^\bullet$ above, by definition of $\mathfrak{D}$ and by definition of the indices:
\begin{equation}
\tiny{
\xymatrix{
&&&&\\
&&&&\\
F^\bullet(N(N(H_2,H_1),N(G_2,G_1)))\ar@/_3pc/[dddd]_-{\Ind_{F^\bullet}^{N(i_2,i_1))}}\ar@/^3pc/[rrrr]^-{\Ind_{F^\bullet}^{N(j_2,j_1))}}&&F^\bullet(\dnc(N(H_2,H_1),N(G_2,G_1)))\ar[ll]^-{s=0}_-{\cong}\ar[rr]^-{s=1}&&F^\bullet(N(H_2,H_1))\ar@/^3pc/[dddd]^-{\Ind_{F^\bullet}^{i_2}}\\
&&&&\\
F^\bullet(\mathfrak{N})\ar[uu]^-{t=0}_-{\cong}\ar[dd]^-{t=1}&&F^\bullet(\mathfrak{D})\ar[uu]^-{t=0}_-{\cong}\ar[dd]^-{t=1}\ar[ll]^-{s=0}_-{\cong}\ar[rr]^-{s=1}&&F^\bullet(\dnc(H_2,H_1))\ar[uu]^-{t=0}_-{\cong}\ar[dd]^-{t=1}\\
&&&&\\
F^\bullet(N(H_2,G_2))\ar@/_3pc/[rrrr]_-{\Ind_{F^\bullet}^{j_2}}&&F^\bullet(\dnc(H_2,G_2))\ar[ll]^-{s=0}_-{\cong}\ar[rr]^-{s=1}&&F^\bullet(H_2)\\
&&&&\\
}}
\end{equation}
where the left vertical diagram gives the right index because of corollary \ref{symmetryDNC}.

The equation \ref{pushforwardfunctIequation} follows from the commutativity of the diagram above..
\end{proof}

\begin{remark}
The hypothesis that $N(j_2,j_1)$ and $N(i_2,i_1)$ are immersions in the last theorem is not at all restrictive. In fact, in the main case we want to consider below (essentially $G_1=H_1$ and $j_1=Id$) it is satisfied. 
\end{remark}

Now, in order to refine our functoriality (based on the examples we want to apply it, as for K-theory and Periodic cyclic theory) we need to ask further properties to our \theory-theory $F^\bullet$ (these properties are quite classic for the examples we have in mind). 

\begin{definition}[vb-groupoid]
Consider the following Lie groupoid morphism
\begin{equation}
\xymatrix{
E\ar[r]\ar@<.5ex>[d]\ar@<-.5ex>[d]&G\ar@<.5ex>[d]\ar@<-.5ex>[d]\\
E_0\ar[r]&G_0.
}
\end{equation}
Then, $E \rightrightarrows E_0$ is said to be a vb-groupoid over $G\rightrightarrows G_0$  if $E \rightarrow G$ and $E_0 \rightarrow G_0$ are vector bundles, and all structural maps are vector bundles morphisms. 
We will refer to $E\rightrightarrows E_0$ as the total groupoid and   $G\rightrightarrows G_0$ as the base groupoid.
\end{definition}

The normal bundle groupoid of an immersion $(G_2,G_1,j)$ is an example of a vb-groupoid.

\begin{definition}[$F^\bullet$-oriented vb-groupoids and $F^\bullet$-Thom isomorphism]\label{Thomproperties}
Let $E\rightrightarrows E_0$ be a vb-groupoid over $G\rightrightarrows G_0$. A $F^\bullet$-orientation for $E\rightrightarrows E_0$ is the choice of a $R$-module isomorphism
\begin{equation}
\xymatrix{
F^\bullet(G)\ar[rr]^-{T_{G,E}}_-\cong && F^\bullet(E)
}
\end{equation}
which is compatible with the restriction in the following sense: for every $K\subseteq G$ closed saturated Lie subgroupoid of $G$, there is an isomorphism 
\begin{equation}
\xymatrix{
F^\bullet(K)\ar[rr]^-{T_{K,E|_K}}_-\cong && F^\bullet(E|_K)
}
\end{equation}
making commutative the square
\begin{equation}\label{Thomnatural}
\xymatrix{
F^\bullet(G)\ar[rr]^-{T_{G,E}}_-\cong \ar[d]_-{r_{K}^{G}}&& F^\bullet(E)\ar[d]^-{r_{E|_K}^{E}}\\
F^\bullet(K)\ar[rr]_-{T_{K,E|_K}}^-\cong && F^\bullet(E|_K)
}
\end{equation}
 These isomorphisms will be called Thom isomorphisms in the sequel (or $F^\bullet$-Thom isomorphisms when such precision is necessary).

We say that $F^\bullet$ satisfies the Thom isomorphism property if given a commutative square of vb-groupoid morphisms 
\begin{equation}
\xymatrix{
E_1\ar[d]&V\ar[l]\ar[d]\\
G&E_2\ar[l]
}
\end{equation}
with $F^\bullet$-orientations $T_{E_1,V}, T_{G,E_1}, T_{E_2,V}$ and $ T_{G,E_2}$ we have that
\begin{equation}\label{ThomThom}
T_{E_1,V}\circ T_{G,E_1}=T_{E_2,V}\circ T_{G,E_2}.
\end{equation}
\end{definition}

\begin{definition}[$F$-oriented Lie groupoid immersion]
Let $(G_2,G_1,j)$ be a Lie groupoid immersion and $F$ an association as above. We say that $j$ is $F^\bullet$-oriented 
if there is an $F^\bullet$ orientation $ T_{G_1,N(G_2,G_1,j)}$ for the normal vb-groupoid $N(G_2,G_1,j)$ over $G_1$. In this case we use the shorthanded notation
$$T_j:=T_{G_1,N(G_2,G_1,j)}$$
for the associated Thom isomorphism.

If $j$ is $F^\bullet$-oriented we define the associated F-pushforward map as the morphism
\begin{equation}
j_!:F^\bullet(G_1)\to F^\bullet(G_2)
\end{equation}
given by the composition 
$$F^\bullet(G_1)\stackrel{T_j}{\longrightarrow}F^\bullet(N(G_2,G_1,j))\xrightarrow{\Ind_{F^\bullet(j)}}F^\bullet(G_2).$$
\end{definition}

\begin{theorem}[Pushforward functoriality II]\label{pushforwardfunctII}
Let $F^\bullet$ be a \theory-theory with the Thom isomorphism property. Suppose we have a commutative diagram of Lie groupoid immersions
\begin{equation}
\xymatrix{
G_1\ar[r]^-{i_1}\ar[d]_-{j_1}&G_2\ar[d]^-{j_2}\\
H_1\ar[r]_-{i_2}&H_2
}
\end{equation}
for which $N(j_2,j_1)$ and $N(i_2,i_1)$ are Lie groupoid immersions. Suppose further that the Lie groupoid immersions $\dnc(j_2,j_1)$ and $\dnc(i_2,i_1)$ are $F$-oriented.
Then the following equality of $R$-modules index morphisms holds
\begin{equation}\label{pushforwardfunctIIequation}
 j_{2!}\circ i_{1!}= i_{2!}\circ j_{1!}
\end{equation}
\end{theorem}

\begin{proof}
By the functorial properties of the deformation to the normal cone, and since 
$N(j_2,j_1)$ and $N(i_2,i_1)$ are assumed to be immersions, the induced map at the level of deformation groupoid 
\begin{equation}
\dnc(G_2,G_1,i_1)\xrightarrow{\dnc(j_2,j_1)}\dnc(H_2,H_1,i_2).
\end{equation}
is a groupoid immersion.
Thus, the iterated deformation groupoid
\begin{equation}
\mathfrak{D}:=\dnc(\dnc(H_2,H_1,i_2),\dnc(G_2,G_1,i_1),\dnc(j_2,j_1)).
\end{equation}
makes sense as a Lie groupoid.
We want to show that the following diagram is commutative

\begin{equation}
\tiny{
\xymatrix@C=1.5pc@R=2.5pc{
&&&&&&\\
F^\bullet(G_1)\ar@{}[rd]|-{{\bf I}}\ar@/_3pc/[ddddd]_-{i_{1!}}\ar@/^3pc/[rrrrr]^-{j_{1!}}\ar[d]_-T^-\cong\ar[r]^-T_-\cong&F^\bullet(N(H_1,G_1))\ar@{}[rrd]|-{{\bf IV}}\ar[d]_-T^-\cong&&F^\bullet(\dnc(H_1,G_1))\ar@{}[rrd]|-{{\bf V}}\ar[d]_-T^-\cong\ar[ll]^-{s=0}_-{\cong}
\ar[rr]^-{s=1}&&F^\bullet(H_1)\ar[d]_-T^-\cong\ar@/^3pc/[ddddd]^-{i_{2!}}\\
F^\bullet(N(G_2,G_1))\ar@{}[rdd]|-{{\bf II}}\ar[r]^-T_-\cong&F^\bullet(N(N(H_2,H_1),N(G_2,G_1)))\ar@{}[rrdd]|-{{\bf A}}&&F^\bullet(\dnc(N(H_2,H_1),N(G_2,G_1)))\ar@{}[rrdd]|-{{\bf A}}\ar[ll]^-{s=0}_-{\cong}
\ar[rr]^-{s=1}&&F^\bullet(N(H_2,H_1)&\\
&&&&&&\\
F^\bullet(\dnc(G_2,G_1))\ar@{}[rdd]|-{{\bf III}}\ar[r]^-T_-\cong\ar[uu]^-{t=0}_-{\cong}\ar[dd]^-{t=1}&F^\bullet(\mathfrak{N})\ar@{}[rrdd]|-{{\bf A}}\ar[uu]^-{t=0}_-{\cong}\ar[dd]^-{t=1}&&F^\bullet(\mathfrak{D})\ar@{}[rrdd]|-{{\bf A}}\ar[uu]^-{t=0}_-{\cong}\ar[dd]^-{t=1}\ar[ll]^-{s=0}_-{\cong}\ar[rr]^-{s=1}&&F^\bullet(\dnc(H_2,H_1))\ar[uu]^-{t=0}_-{\cong}\ar[dd]^-{t=1}&\\
&&&&&&\\
F^\bullet(G_2)\ar@/_3pc/[rrrrr]_-{j_{2!}}\ar[r]_-T^-\cong&F^\bullet(N(H_2,G_2))&&F^\bullet(\dnc(H_2,G_2))\ar[ll]^-{s=0}_-{\cong}\ar[rr]^-{s=1}&&F^\bullet(H_2)&\\
&&&&&&\\
}}
\end{equation}
As visually sketch in the diagram above, we have separated the diagram in smaller squares. For keep notation shorts we denoted by $T$ all the involved Thom isomorphisms (that we will detail further below) and by $s=0,1$, $t=0,1$ the canonical restrictions as for theorem \ref{pushforwardfunctI}.

Now, the diagrams marked with an $\bf A$ correspond exactly to theorem \ref{pushforwardfunctI} and hence they are commutative.

{\bf Diagrams II and III:} By hypothesis, the map $\dnc(j_2,j_1)$ is an $F$-oriented immersion and hence there is Thom isomorphism
$$T_{\dnc(j_2,j_1)}:F^\bullet(\dnc(G_2,G_1))\stackrel{\cong}{\longrightarrow} F^\bullet(\mathfrak{N}).$$
By naturality of the Thom isomorphism, property \ref{Thomnatural} above, diagrams {\bf II} and {\bf III} are commutative.

{\bf Diagrams IV and V:} The commutativity of diagrams {\bf IV} and {\bf V} follow by the exact same arguments as for the diagrams {\bf II} and {\bf III} but applying it to the immersion $\dnc(i_2,i_1)$ together with the fact that 
\begin{equation}
\dnc(N(H_2,H_1),N(G_2,G_1))=\dnc(N(j_2,j_1))
\end{equation}
and 
\begin{equation}
N(\dnc(i_2,i_1))
\end{equation}
are canonically isomorphic groupoids by corollary \ref{symmetryDNC}.

{\bf Diagram I:} This diagram's commutativity states as the equality
\begin{equation}\label{ThomThomnormalequation}
T_{N(j_2,j_1)}\circ T_{i_1}=T_{N(i_2,i_1)}\circ T_{j_1}
\end{equation}
up to the isomorphism of theorem \ref{symmetryNBthm}. This equality holds by Thom property \ref{ThomThom} above.

This concludes the theorem.
\end{proof}

Now, we have been assuming the condition that $N(j_2,j_1)$ and $N(i_2,i_1)$ are both immersions. As we will see now this condition is not restrictive for the cases we will work out below. Let us start by introducing the notion of immersion of vector bundle: a vector bundle map $E_1 \rightarrow E_2$ is said to be a vb-immersion if it is fiberwise injective and it covers an immersion of smooth manifold.
\begin{proposition}\label{prop-2immer}
Consider a commutative square consisting only of immersions of smooth manifolds:
\begin{equation}\label{2-immersion}
\xymatrix{
M_1\ar[r]^{i_1}\ar[d]_-{j_1}&M_2\ar[d]^-{j_2}\\
N_1\ar[r]_{i_2}&N_2
}
\end{equation}
     
    The following assertions are equivalent:
    \begin{enumerate}
        \item[$(1)$] The normal differential $N(j_2,j_1): N(i_1) \rightarrow N(i_2)$ is a vb-immersion.
        \item[$(2)$] The normal differential $N(i_2,i_1): N(j_1) \rightarrow N(j_2)$ is a vb-immersion.
        \item[$(3)$] For each $m \in M_1$, the following sequence of vector spaces is exact:
        $$
        0 \rightarrow T_mM_1 \xrightarrow{(j_{1*},i_{1*})} T_{j_1(m)}N_1 \times T_{i_1(m)}M_2 \xrightarrow{i_{2*} - j_{2*}} T_nN_2
        $$
        with $n= j_2 \circ i_1(m) = i_2 \circ j_1(m)$.
    \end{enumerate}
\end{proposition}
\begin{remark} 
The condition (2) reminds the "good pair condition" from \cite{BCH16} appendix A. In the case of embedding, it corresponds to the clean instersection condition $TN_1 \cap TM_2 = TM_1$.
\end{remark}
\begin{proof}
    By symmetry, we just need to show $(1)\Leftrightarrow (3)$.
    \par
    $(1) \Rightarrow (3)$: 
    Let $\delta \in T_{i_1(m)} M_2$ such that 
    $
    j_{2*}(\delta)$ belongs to $ i_{2*}(T_{j_1(m)} N_1)$, or equivalently $[\delta] \in \ker N(j_2,j_1)$.
    Since $N(j_2,j_1)$ is a vb-immersion, one must have $\delta = i_{1*}(\varepsilon)$ for some unique $\varepsilon$ in $T_m M_1$. 
    Put $\gamma= j_{1*}(\varepsilon)$ in $T_{j_1(m)} N_1$, hence $j_{2*}(\delta) = i_{2*}(\gamma)$.
    In particular, $(\delta,\gamma)$ lies inside $Im(j_{1*}, i_{1*})$.
    \par
    $(3) \Rightarrow (1)$:
    Let $[\delta] \in \ker N(j_2,j_1)$, then by lifting from $N(j_2,j_1)$ to $j_1 \times_{j_2} j_{2*}$ this amounts to $j_{2*}(\delta) \in i_{2*}(TN_1)$. Let $\gamma \in T_{j_1(m)}N_1$ such that $j_{2*}(\delta)=i_{2*}(\gamma)$. Thus, by exactness, $(\delta,\gamma)$ are in $Im(j_{1*}, i_{1*})$, in particular $[\delta]=0$.
\end{proof}
\begin{remark}
    The inclusion $Im(j_{1*}, i_{1*}) \subseteq \ker(i_{2*} - j_{2*})$ always holds (direct consequence of the commutativity of the square~(\ref{2-immersion})), but the reverse inclusion is not automatic. For example, consider the following situation:
  \begin{equation}
  \begin{split}
   \xymatrix{
    \mathbb{R}\ar[r]\ar[d]&\mathbb{R}^3\ar[d]\\
    \mathbb{R}^5\ar[r]&\mathbb{R}^6
    }
  \end{split}
\end{equation}
    where all the maps are the inclusions of the corresponding first coordinates. Then applying the normal functor leads to the following vector bundle map
    $$
    N(\mathbb{R}^5, \mathbb{R}^4) = \mathbb{R}^4 \times \mathbb{R} \rightarrow 
    N(\mathbb{R}^6, \mathbb{R}^3) = \mathbb{R}^3 \times \mathbb{R}^3
    $$
    which cannot be a vb-immersion because of a dimension issue.
\end{remark}
As an application of the previous proposition, we obtain
\begin{corollary}\label{cor-monosquare-specialcase}
   Suppose we have a commutative diagram of Lie groupoid immersions
\begin{equation}
\begin{split}
  \xymatrix{
  G_1\ar[r]^-{i_1}\ar[d]_-{Id_{G_1}}&G_2\ar[d]^-{j_2}\\
  G_1\ar[r]_-{i_2}&H_2
}
\end{split}
\end{equation}
    then the (equivalent) assertions of proposition~\ref{prop-2immer} hold where the respective normal differential are Lie groupoid immersions in condition (1) and (2).
\end{corollary}
\begin{proof}
A straightforward computation allows to verify condition (3) of proposition~\ref{prop-2immer}.
\end{proof}

\section{Exemples, applications}

\subsection{Equivariant wrong way $F^\bullet$-functoriality for \'etale Lie groupoids}\label{subsectionFetalefunctoriality}

\begin{definition}
Let $H\rightrightarrows H^{(0)}$ be a Lie groupoid. It is say to be \'etale if the source and target maps are both local diffeomorphisms.
\end{definition}

Let $f:G_1\longrightarrow G_2$ be a Lie groupoid morphism where $G_1$ is \'etale. Associated to such a morphism we can consider the following Lie groupoid morphism
\begin{equation}
\tilde{f}:G_1\longrightarrow G_2\times (G_1^{(0)}\times G_1^{(0)})
\end{equation}
where $\tilde{f}:=(f, (s,t))$ ($s$ and $t$, source and target maps of $G_1$), and 
where at the right we have the product groupoid of $G_2$ and the pair groupoid of $G_1^{(0)}$. 

Because $G_1$ is \'etale we have that $\tilde{f}$ is a Lie groupoid immersion (not necessarilly injective). But even more important for applications, it is Morita equivalent to the original morphism $f$ in the sense that the following Lie groupoid morphism diagram commutes
\begin{equation}
\xymatrix{
G_1 \ar[r]^-f\ar[rd]_-{\tilde{f}}& G_2\\
& G_2\times (G_1^{(0)}\times G_1^{(0)})\ar[u]_-{m}^-\sim
}
\end{equation}
where $m$ is the canonical Morita equivalence induced by the canonical Morita equivalence between the pair groupoid of $G_1^{(0)}$ and a point. Hence, any Lie groupoid morphism as above with $G_1$ \'etale can be transformed, up to Morita equivalence, into a Lie groupoid immersion. 

Now, we want to consider a more general situation. Consider $f:G_1\longrightarrow G_2$ as above with the extra data of an external Lie groupoid $G\rightrightarrows G_0$   acting on the groupoids $G_1$ and $G_2$, and such that the morphism $f$ is $G$-equivariant. In this situation we have a commutative diagram of Lie groupoid morphism between the respective associated semi-direct product groupoids
\begin{equation}\label{tildeetale}
\xymatrix{
G_1\rtimes G \ar[r]^-f\ar[rd]_-{\tilde{f}}& G_2\rtimes G\\
& (G_2\times (G_1^{(0)}\times G_1^{(0)}))\rtimes G\ar[u]_-{m}^-\sim
}
\end{equation}
where we still denote by $f$ and $\tilde{f}$ the associated maps to not load more the notation.

\begin{remark}
In the appendix \ref{subsectiongrpdactinggrpd} we review the basic definition of a groupoid acting on another groupoid and the associated semi-direct product. In particular, it is important to remark that the notation $H\rtimes G$ for the semi-direct product groupoid associated to a groupoid action of $G$ on $H$ includes already the respective fiber product domains. Indeed, by definition of a groupoid action $H$ (and $H_0$) comes with a moment map over $G_0$ and $H\rtimes G:=H\times_{G_0}G$. So for example, in the case above $(G_2\times (G_1^{(0)}\times G_1^{(0)}))\rtimes G=(G_2\times (G_1^{(0)}\times G_1^{(0)}))\times_{G_0} G=G_2\times_{G_0} (G_1^{(0)}\times_{G_0} G_1^{(0)})\times_{G_0} G$ which is Morita equivalent to $G_2\times_{G_0} G_0 \times_{G_0} G=G_2 \times_{G_0} G$. 
\end{remark}

Given a \theory-theory $F^\bullet$ as in the previous section, we want to associate a wrong way map $f_!$ to $f$. Now, since we know how to do it (under $F^\bullet$-oriented conditions) for $\tilde{f}$, we would have a good definition if our theory satisfies a good Morita invariance property, this is exactly the content of the following axiom/definition

\begin{definition}[$F^\bullet$-Morita invariance]\label{defMoritainv}
Consider a \theory-theory $F^\bullet$ with the Thom isomorphism property as above. We say that it satisfies Morita invariance if given a Morita equivalence of Lie groupoids
\begin{equation}
\xymatrix{
H\ar[r]^-m_-\sim & H'
}
\end{equation}
there is an associated isomorphism
\begin{equation}
\xymatrix{
F^\bullet(H)\ar[r]^-{F(m)}_-\cong & F^\bullet(H')
}
\end{equation}
that satisfies:
\begin{itemize}
\item (Functoriality) If 
\begin{equation}
\xymatrix{
H\ar[r]^-m_-\sim & H'\ar[r]^-{m'}_-\sim & H""
}
\end{equation}
is a composition of Lie groupoid Morita equivalences, then 
\begin{equation}
F(m'\circ m)=F(m')\circ F(m).	
\end{equation}
\item (Naturality)
If  
\begin{equation}
\xymatrix@C=3pc{
H\ar[r]^-{m_H}_-\sim & H'\\
K\ar[u]\ar[r]^-{m_K}_-\sim & K'\ar[u]
}
\end{equation}
is a commutative diagram of Lie groupoid morphisms with $m_H,m_K$ Morita equivalences and vertical maps given by inclusions of closed subgroupoids, then the following diagram of $R$-module morphisms
\begin{equation}
\xymatrix@C=3pc{
F^\bullet (H)\ar[d]_-{r_{K}^{H}}\ar[r]^-{F(m_H)}_-\cong & F^\bullet (H')\ar[d]^-{r_{K'}^{H'}}\\
F^\bullet (K)\ar[r]_-{F(m_K)}^-\cong & F^\bullet (K')
}
\end{equation}
is commutative.
\item (Thom compatibility)
If  
\begin{equation}
\xymatrix@C=3pc{
H\ar[r]^-{m_H}_-\sim & H'\\
E\ar[u]\ar[r]^-{m_E}_-\sim & E'\ar[u]
}
\end{equation}
is a commutative diagram of Lie groupoid morphisms with $m_H,m_E$ being Morita equivalences and the vertical maps are $F^\bullet$-oriented vb-groupoids projections, then the following diagram of $R$-module isomorphisms
\begin{equation}
\xymatrix{
F^\bullet (H)\ar[d]_-{T_{H,E}}\ar[r]^-{F(m_H)}_-\cong & F^\bullet (H')\ar[d]^-{T_{H',K'}}\\
F^\bullet (E)\ar[r]_-{F(m_E)}^-\cong & F^\bullet (E')
}
\end{equation}
is commutative.
\end{itemize}

\end{definition}

\begin{definition}
Consider a \theory-theory $F^\bullet$ with the Thom isomorphism property and satisfying Morita invariance. Let $G$ be a Lie groupoid. Let $f:G_1\longrightarrow G_2$ be a $G$-equivariant Lie groupoid morphism between two \'etale Lie groupoids. We say that it is $F^\bullet$-oriented if the associated immersion $\tilde{f}$ between the asociated semi-direct product groupoids is an $F^\bullet$-oriented immersion (definition above). Under this assumption we can define
\begin{equation}
f_!:F^\bullet(G_1\rtimes G)\to F^\bullet(G_2\rtimes G) 
\end{equation}
as the composition of 
\begin{equation}
\tilde{f}_!:F^\bullet(G_1\rtimes G)\to F^\bullet((G_2\times (G_1^{(0)}\times G_1^{(0)}))\rtimes G),
\end{equation}
defined for immersions as in the previous section, followed by the induced isomorphism
\begin{equation}
F(m):F^\bullet((G_2\times (G_1^{(0)}\times G_1^{(0)}))\rtimes G)\stackrel{\cong}{\longrightarrow} F^\bullet(G_2\rtimes G)
\end{equation}
induced from the canonical Morita equivalence
\begin{equation}
\xymatrix{
(G_2\times (G_1^{(0)}\times G_1^{(0)}))\rtimes G \ar[r]^-m_-{\sim}& G_2\rtimes G
}
\end{equation}
\end{definition}

For getting closer to the applications we have in mind, pushforward functoriality in K-theory and in delocalised orbifold cohomology, we need to make an extra assumption on our $F^\bullet$-theory, that will of course, as we will see below, be satisfied in the geometric contexts we will treat below.

\begin{definition}\label{defOdefproperty}
Consider a \theory-theory $F^\bullet$ with the Thom isomorphism property and satisfying Morita invariance. We say that it satisfies the orientation deformation property if given a commutative diagram 
\begin{equation}
\xymatrix{
G_1\ar[r]^-{i_1}\ar[d]_-{j_1}&G_2\ar[d]^-{j_2}\\
H_1\ar[r]_-{i_2}&H_2
}
\end{equation}
of Lie groupoid $F^\bullet$-oriented immersions such that $\dnc(i_2,i_1)$ and $\dnc(j_2,i_1)$ are Lie groupoids immersions one has that these deformation immersions, $\dnc(i_2,i_1)$ and $\dnc(j_2,i_1)$, are as well $F^\bullet$-oriented.
\end{definition}

The following theorem is a corollary of theorem \ref{pushforwardfunctII}.

\begin{theorem}[Equivariant wrong way $F^\bullet$-functoriality for \'etale Lie groupoids]\label{thmFetale}
Consider a \theory-theory $F^\bullet$ with the Thom isomorphism property, satisfying Morita invariance and with the orientation deformation property. Let $G$ be a Lie grouppoid.
Let 
\begin{equation}\label{etaletriangle}
\xymatrix{
G_1\ar[rd]_-{f_2}\ar[r]^-{f_1}&G_2\ar[d]^-{g_2}\\
&H_2
}
\end{equation}
be a commutative diagram of Lie groupoid $F^\bullet$-oriented $G$-equivariant morphisms between \'etale groupoids. Then, the following equality of $R$-module morphisms holds:
\begin{equation}
g_{2!}\circ f_{1!}=f_{2!}
\end{equation}
\end{theorem}

\begin{proof}
Let us fix some extra notation, for $i=1,2, $ we will denote by $M_i$ the base of the groupoid $G_i$ and by $N_2$ the base of the groupoid $H_2$.

Associated to diagram (\ref{etaletriangle}) we have the following commutative diagram of Lie groupoid immersions
\begin{equation}
\xymatrix{
G_ 1\ar[rr]^-{\tilde{f_1}}\ar[d]_-{Id_{G_1}}&&G_2\times (M_1)^2\ar[d]^-{\tilde{\tilde{g_2}}}\\
G_1\ar[rr]_-{\tilde{\tilde{f_2}}}&&H_2\times (M_1)^2\times (M_2)^2
}
\end{equation}
where $\tilde{f_1}$ is given as (\ref{tildeetale}) above, where, for $(\gamma,(x,y))\in G_2\times (M_1)^2$
\begin{equation}
\tilde{\tilde{g_2}}(\gamma,(x,y))=(g_2(\gamma),(x,y),(s_{G_2}(\gamma),t_{G_2}(\gamma)))
\end{equation}
and where, for $\eta\in G_1$
\begin{equation}
\tilde{\tilde{f_2}}(\eta)=(f_2(\eta),(s_{G_1}(\eta),t_{G_1}(\eta)),(f_1^{(0)}(s_{G_1}(\eta)),f_1^{(0)}(t_{G_1}(\eta)))).
\end{equation}
Now, the double tilde maps $\tilde{\tilde{f_2}}$ and $\tilde{\tilde{g_2}}$ are not exactly the ones used to define  the respective pushforward maps but they are K-equivalent to them in the sense that they fit in the following commutative diagram
\begin{equation}
\xymatrix{
G_ 1\ar[rr]^-{\tilde{f_1}}\ar[d]_-{Id_{G_1}}&&G_2\times (M_1)^2\ar[d]^-{\tilde{\tilde{g_2}}}\ar[r]^-\sim&G_2\ar[d]^-{\tilde{g_2}}\\
G_1\ar[rr]_-{\tilde{\tilde{f_2}}}\ar[d]_{Id_{G_1}} &&H_2\times (M_1)^2\times  (M_2)^2\ar[r]^-\sim\ar[d]_-\sim& H_2\times (M_2)^2\ar[d]^-\sim\\
G_1\ar[rr]_-{\tilde{f_2}}&&H_2\times (M_1)^2\ar[r]^-\sim& H_2
}
\end{equation}
where the maps labeled with a $\sim$ are canonical Morita equivalence that induce canonical isomorphisms (which are natural and Thom compatible in the sense if definition \ref{defMoritainv}) in $K$-theory.
So, in order to obtain the theorem's conclusion as an immediate corollary of the Morita invariance property and of theorem \ref{pushforwardfunctII} we just need to check theorem's \ref{pushforwardfunctII} hypothesis for the diagram 
\begin{equation}\label{auxdiagetale}
\xymatrix{
G_ 1\rtimes G\ar[rr]^-{\tilde{f_1}\times Id_G}\ar[d]_-{Id}&&(G_2\times (M_1)^2)\rtimes G\ar[d]^-{\tilde{\tilde{g_2}}\times Id_G}\\
G_1\rtimes G\ar[rr]_-{\tilde{\tilde{f_2}}\times Id_G}&&(H_2\times (M_1)^2\times (M_2)^2)\rtimes G
}
\end{equation}
 but these are  satisfied by corollary \ref{cor-monosquare-specialcase} and by hypothesis on the present theorem.
\end{proof}

\subsection{Wrong way functoriality for equivariant (twisted) orbifold K-theory}\label{subsectionOrbifold}

We can now come to our main explicit example of an \theory-theory. For a Lie groupoid $H$ we consider in this section the theory
\begin{equation}
K^*(H):=K_*(C^*(H))
\end{equation}
given by taking the K-theory group\footnote{For the purpose of this paper $K_*=K_0\oplus K_1$. For further applications it will be necessary to keep into account the degree and the shifting of degrees when taking the pushforward maps but for the moment we do not need this.} of its maximal $C^*$-algebra.

Now, by classical K-theoretical functorial properties, we do have that this theory $K^*(H):=K_*(C^*(H))$ is a \theory-theory in the sense required above in section \ref{sectiondefindices}. Another important point is that two Morita equivalent Lie groupoids give rise to strongly Morita equivalent $C^*$-algebras and hence to isomorphic $K$-theory groups (see \cite{Willtool} theorem 2.70 for the Morita invariance and \cite{Ror} for classic K-theory properties of $C^*$-algebras).

{\bf On the Thom isomorphism property:} Given a Lie groupoid $H$ and $vb$-groupoid $E$ over $H$ we would expect that under some good $spin^c$-properties we would have a Thom isomorphism
\begin{equation}
K^*(H)\stackrel{\cong}{\longrightarrow}K^*(E).
\end{equation}
However, as far as the authors are aware, this is not known or at least not properly established in the litterature, at least for the general case of vb-groupoids. Even for other (co)homological theories for (differentiable stacks) Lie groupoids this seems to be unknown as well. Now, there is however a good news. In K-theory (and even in KK-theory) th Thom isomorphism is known for the case in which one has equivariant $spin^c$-vector bundles with the respect to an action of a Lie groupoid. So, if $H\rightrightarrows H^{(0)}$ is a Lie groupoid and $E\to P$ a $H$-equivariant $spin^c$ $H$-vector bundle then there is a Thom isomorphism
\begin{equation}\label{KequivariantThom}
\xymatrix{
K^*(P\rtimes H)\ar[rr]^-{T_{P\rtimes H,E\rtimes H}}_-{\cong} && K^*(E\rtimes H)
}
\end{equation}
satisfying the abstract Thom properties of the previous section. For more details on this and on its twisted K-theoretical version the reader can see the appendix A2 in \cite{CaWangAENS}, \cite{CW} and \cite{LeGall}.

We will then restrict ourselves to a class of groupoids for which the Thom isomorphism property holds because they will fit in the case of groupoid equivariant Thom isomorphism as just explained above. Also, we will need that the condition on being K-oriented is not to restrictive. All this will be fullfilled if we play with orbifold groupoids, which by the way were our original motivation for this article.

\begin{definition}
Let $\mathfrak{X}\rightrightarrows \mathfrak{X}^{(0)}$ be a Lie groupoid with source and target maps $s$ and $t$ respectively. The groupoid is called an orbifold groupoid if it is \'etale ($s$ and $t$ are local diffeomorphisms) and proper, meaning that the map
\begin{equation}
(s,t): \mathfrak{X}\to \mathfrak{X}^{(0)}\times \mathfrak{X}^{(0)}
\end{equation}
is proper.
\end{definition}

Let $\mathfrak{X}_i \rightrightarrows M_i$ be two orbifold groupoids. Consider a Lie groupoid morphism
\begin{equation}
f:\mathfrak{X}_1\longrightarrow \mathfrak{X}_2.
\end{equation}

Associated to such an $f$ we have the following smooth map
\begin{equation}\label{deftildef}
\xymatrix{
\mathfrak{X}_1\ar[r]^-{\tilde{f}}& \mathfrak{X}_2 \times (M_1\times M_1),
}
\end{equation}
which is a (not always injective) immersion of Lie groupoids, where $\tilde{f}=(f, (s,t))$ as in the previous section. In particular, we have its associated normal bundle groupoid
\begin{equation}
N(\mathfrak{X}_2 \times (M_1\times M_1),\mathfrak{X}_1)\rightrightarrows N(M_2\times M_1, M_1).
\end{equation} 

On the other hand, we have the semi-direct product groupoid
\begin{equation}
(TM_1\ltimes f_0^*TM_2)\rtimes \mathfrak{X}_1 \rightrightarrows 
f_0^*TM_2
\end{equation}
described expicitly in the appendix, see equation (\ref{normalgrpdorbimorphism}).

We have the following result.

\begin{proposition}\label{normalgrpdiso}
Let $\mathfrak{X}_i \rightrightarrows M_i$ be two orbifold groupoids. Consider a Lie groupoid morphism
\begin{equation}
f:\mathfrak{X}_1\longrightarrow \mathfrak{X}_2.
\end{equation}
There is a Lie groupoid isomorphism
\begin{equation}
\xymatrix{
N(\mathfrak{X}_2 \times (M_1\times M_1),\mathfrak{X}_1) \ar@<.5ex>[d]\ar@<-.5ex>[d] \ar[r]^-\phi_-\cong & (TM_1\ltimes f_0^*TM_2)\rtimes \mathfrak{X}_1 \ar@<.5ex>[d]\ar@<-.5ex>[d] \\
N(M_2\times M_1, M_1) \ar[r]^-\cong_-{\phi_0} & f_0^*TM_2
}
\end{equation}
\end{proposition}

\begin{proof}
Since $\mathfrak{X}_i$ is an \'etale groupoid we have vector bundle isomorphisms
\begin{equation}
T\mathfrak{X}_i\cong t^*TM_i\,\,\text{and}\,\, T\mathfrak{X}_i\cong s^*TM_i.
\end{equation}
The first isomorphism is explicitly given by the differential of the target map and the second by the differential of the source map.
By definition, $N(\mathfrak{X}_2 \times (M_1\times M_1),\mathfrak{X}_1)$ is the quotient of 
$$(f,(s,t))^*T(\mathfrak{X}_2\times M_1\times M_1)$$
by 
$$T\mathfrak{X}_1\cong t^*TM_1$$
where $T\mathfrak{X}_1$ is send into $(f,(s,t))^*(T(\mathfrak{X}_2)\times M_1\times M_1)$ by the differential of $(f,(s,t))$. Now, a standard computation using the fact that $f$ is a Lie groupoid morphism and classic pullback functoriality properties gives
\begin{equation}
\begin{array}{rcl}
(f,(s,t))^*T(\mathfrak{X}_2\times M_1\times M_1)&\cong& f^*T\mathfrak{X}_2\oplus s^*TM_1\oplus t^*TM_1
\\
&\cong& f^*t^*TM_2\oplus t^*TM_1\oplus t^*TM_1 
\\
&\cong& t^*(f_0^*TM_2\oplus TM_1\oplus TM_1)
\end{array}
\end{equation}
that induces the following isomorphism
\begin{equation}
N(\mathfrak{X}_2 \times (M_1\times M_1),\mathfrak{X}_1)\cong t^*\left(\frac{(f_0,\Delta)^*(TM_2\oplus TM_1\oplus TM_1)}{TM_1}\right),
\end{equation}
where $\Delta$ stands for the diagonal map.
Now, since 
$$\frac{(f_0,\Delta)^*(TM_2\oplus TM_1\oplus TM_1)}{TM_1}$$
is isomorphic to $f_0^*TM_2\oplus TM_1$ by sending a class $[W,V_1,V_2]\to (W,V_1-V_2)$, we obtain an isomorphism 
$$t^*\left(\frac{(f_0,\Delta)^*(TM_2\oplus TM_1\oplus TM_1)}{TM_1}\right)
\cong t^*(f_0^*TM_2\oplus TM_1)=:(TM_1\ltimes f_0^*TM_2)\rtimes \mathfrak{X}_1.$$
Let us explicit the formula of the global isomorphism by unfolding all the previous computations. We define
\begin{equation}
\xymatrix{
N(\mathfrak{X}_2 \times (M_1\times M_1),\mathfrak{X}_1)\ar[r]^\phi & (TM_1\ltimes f_0^*TM_2)\rtimes \mathfrak{X}_1
}
\end{equation}
as follows. For $g\in \mathfrak{X}_1$, $V_1\in T_{t(g)}M_1$, $V_2\in T_{s(g)}M_1$ and $W\in T_{f(g)}\mathfrak{X}_2$ we let
\begin{equation}
\phi([g,V_1,V_2,W])=((t(g),V_1-(t_*\circ s_*^{-1})(V_2), t_*W-(f_0)_*((t_*\circ s_*^{-1})(V_2))),g).
\end{equation}
This is indeed the composition of all the above computations and give thus an isomorphism of vector bundles over $\mathfrak{X}_1$. But moreover, we consider as well
\begin{equation}
\xymatrix{
N(M_2\times M_1, M_1)\ar[r]^-{\phi_0} &f_0^*TM_2
}
\end{equation}
the vector bundle isomorphism given by
\begin{equation}
\phi_0([x,V,W_0])=(x,W_0-(f_0)_*(V))
\end{equation}
for $(V,W_0)\in T_xM_1\times T_{f_0(x)}M_2$. It is now a straightforward computation to chech that these isomorphims fit together into a Lie groupoid isomorphism
\begin{equation}
\xymatrix{
N(\mathfrak{X}_2 \times (M_1\times M_1),\mathfrak{X}_1) \ar@<.5ex>[d]\ar@<-.5ex>[d] \ar[r]^-\phi_-\cong & (TM_1\ltimes f_0^*TM_2)\rtimes \mathfrak{X}_1 \ar@<.5ex>[d]\ar@<-.5ex>[d] \\
N(M_2\times M_1, M_1) \ar[r]^-\cong_-{\phi_0} & f_0^*TM_2.
}
\end{equation}
\end{proof}

Now, under the assumption that $f:\mathfrak{X}_1\longrightarrow \mathfrak{X}_2$ as above is a $G$-equivariant Lie groupoid morphism, with $G\rightrightarrows G_0$ a Lie groupoid, we have an associated smooth map
\begin{equation}\label{Gftilde}
\mathfrak{X}_1\rtimes G\xrightarrow{\tilde{f}\times Id_G} (\mathfrak{X}_2 \times (M_1\times M_1))\rtimes G,
\end{equation}
which is a (not injective) immersion of Lie groupoids, where $\tilde{f}=(f, (s,t))$ as above. A direct computation shows that the associated normal bundle groupoid
\begin{equation}
N((\mathfrak{X}_2 \times (M_1\times M_1))\rtimes G,\mathfrak{X}_1\rtimes G)\rightrightarrows N(M_2\times M_1, M_1)
\end{equation}
is canonically isomorphic to
\begin{equation}
N(\mathfrak{X}_2 \times (M_1\times M_1),\mathfrak{X}_1)\rtimes G\rightrightarrows N(M_2\times M_1, M_1)
\end{equation}
where the action of $G$ on the normal groupoid associated to $\tilde{f}$ is obtained by functoriality of the normal bundle construction.

The following proposition is a step by step adaptation of the proposition above to the equivariant case.

\begin{proposition}\label{normalgrpdisoequivariant}
Let $\mathfrak{X}_i \rightrightarrows M_i$ be two orbifold groupoids. Consider a Lie groupoid morphism
\begin{equation}
f:\mathfrak{X}_1\longrightarrow \mathfrak{X}_2.
\end{equation}
Assume there is $G$ a Lie groupoid acting on the groupoids above in such a way that $f$ is $G$-equivariant Lie groupoid morphism.
Then, there is a Lie groupoid isomorphism
\begin{equation}
\xymatrix{
N(\mathfrak{X}_2 \times (M_1\times M_1),\mathfrak{X}_1))\rtimes G \ar@<.5ex>[d]\ar@<-.5ex>[d] \ar[r]^-\cong & ((TM_1\ltimes f_0^*TM_2)\rtimes \mathfrak{X}_1)\rtimes G \ar@<.5ex>[d]\ar@<-.5ex>[d] \\
N(M_2\times M_1, M_1) \ar[r]_-\cong & f_0^*TM_2
}
\end{equation}
\end{proposition}

\begin{definition}
Let $f:\mathfrak{X}_1\longrightarrow \mathfrak{X}_2$ be a $G$-equivariant Lie groupoid morphism between two orbifold groupoids as above. We say that $f$ is $G$-$K$-oriented if the vector bundle
$$T^*M_1\oplus f_0^*TM_2\to M_1$$
admits a $\mathfrak{X}_1\rtimes G$-$spin^c$-structure. In that case, set  $$K_G^*(\mathfrak{X}:)=K^*(\mathfrak{X}\rtimes G),$$
and define
\begin{equation}
f_!: K_G^*(\mathfrak{X}_1)\to K_G^*(\mathfrak{X}_2) 
\end{equation}
as the following composite:
\begin{enumerate}
\item the Thom isomorphism (described in appendix~\ref{appendixThomfdef})
\begin{equation}\label{Thomftilde}
\xymatrix{
T_{\tilde{f}}:K^*_G(\mathfrak{X}_1)\ar[r]^-\cong & K^*_G(N(\tilde{f})) 
},
\end{equation} 
\item the deformation index morphism associated to the groupoid immersion (\ref{Gftilde})
\begin{equation}
K_G^*(N(\tilde{f}))\xrightarrow{\quad\Ind_{K_ G^*}^{\tilde{f}}\quad } K_G^*(\mathfrak{X}_2 \times (M_1\times M_1)),
\end{equation}
\item the natural isomorphism
\begin{equation}
\xymatrix{
K_G^*(\mathfrak{X}_2 \times (M_1\times M_1))\ar[rr]^-{K_ G^*(m)}_-\cong&& K_G^*(\mathfrak{X}_2)
}
\end{equation}
induced by the canonical Morita equivalence 
$$
\xymatrix{
(\mathfrak{X}_2 \times (M_1\times M_1))\rtimes G \ar[r]^-m_-\sim & \mathfrak{X}_2\rtimes G.
}
$$
\end{enumerate}
\end{definition}

Let us state our main theorem for equivariant orbifold K-theory.

\begin{theorem}[Wrong way functoriality for equivariant orbifold K-theory]
Let $G$ be a Lie groupoid.
Let 
\begin{equation}\label{orbifoldsquare}
\xymatrix{
\mathfrak{X}_ 1\ar[r]^-{f}\ar[dr]_-{h}&\mathfrak{X}_2\ar[d]^-{g}\\
&\mathfrak{X}_3
}
\end{equation}
be a commutative diagram $K$-oriented $G$-equivariant Lie groupoid morphisms between orbifold groupoids. Then, the following diagram of abelian group morphisms commutes:
\begin{equation}
\xymatrix{
K^*_G(\mathfrak{X}_ 1)\ar[r]^-{f_!}\ar[dr]_-{h_!}&K^*_G(\mathfrak{X}_2)\ar[d]^-{g_!}\\
&K^*_G(\mathfrak{X}_3)
}
\end{equation}
\end{theorem}

\begin{proof}
Let us fix some extra notation, for $i=1,2,3, $ we will denote by $M_i$ the base of the groupoid $\mathfrak{X}_i$.
Associated to diagram (\ref{orbifoldsquare}) we have the following commutative diagram of Lie groupoid immersions
\begin{equation}
\xymatrix{
\mathfrak{X}_ 1\ar[rr]^-{\tilde{f}}\ar[d]_-{Id}&&\mathfrak{X}_2\times (M_1)^2\ar[d]^-{\tilde{\tilde{g}}}\\
\mathfrak{X}_1\ar[rr]_-{\tilde{\tilde{h}}}&&\mathfrak{X}_3\times (M_1)^2\times (M_2)^2
}
\end{equation}
where $\tilde{f}$ is given as (\ref{deftildef}) above, and where, up to evident permutation of factors, 
\begin{equation}\label{tildetilde2}
\tilde{\tilde{g}}=\tilde{g} \times Id_{(M_1^2)}
\end{equation}
and
\begin{equation}\label{tildetilde3}
\tilde{\tilde{h}}=(\tilde{h}, (f_0\circ s_1,f_0\circ t_1)).
\end{equation}

For make easier the lecture with respect to the previous sections, we denote 
\begin{equation}
i_1=\tilde{f}\times Id_G, j_1=Id_{\mathfrak{X}_1}\times Id_G, i_2=\tilde{\tilde{h}}\times Id_G, j_2=\tilde{\tilde{g}}\times Id_G.
\end{equation}

In order to be able to obtain this theorem's conclusion as a corollary of theorem \ref{thmFetale} or as direct corollary of theorem \ref{pushforwardfunctII} we need to show that, in the context of the statement of the present theorem we have the same properties and hypothesis, that is:

\begin{enumerate}\label{2hypothesis}
\item The Thom isomorphism property holds, either as in the general case of the Thom property \ref{ThomThom} or as its particular case used for concluding theorem \ref{pushforwardfunctII} proof, that is, equation (\ref{ThomThomnormalequation}) that writes down as 
\begin{equation}
T_{N(j_2,j_1)}\circ T_{i_1}=T_{N(i_2,i_1)}
\end{equation}
in the present case since $j_1=Id$.
\item The orientation deformation property in $K$-theory for this case, that is, we need to justify that
both $\dnc(i_2,i_1)$ and $\dnc(j_2,j_1)$ are K-oriented Lie groupoid immersions. 
\end{enumerate} 

Following proposition \ref{normalgrpdisoequivariant} together with (\ref{tildetilde2}) and (\ref{tildetilde3}), one obtains by a straightforward computation an explicit description for $N(j_2,j_1)$ and $N(i_2,i_1)$ (up to these isomorphisms), indeed, 
\begin{equation}\label{normalj}
\xymatrix@C=1.5pc{
((TM_1\ltimes f_0^*TM_2)\rtimes \mathfrak{X}_1)\rtimes G \ar@<.5ex>[d]\ar@<-.5ex>[d] \ar[rr]^-{N(j_2,j_1)}&&((TM_1\oplus f_0^*TM_2)\ltimes (f_0^*TM_2\oplus h_0^*TM_3)\rtimes \mathfrak{X}_1)\rtimes G\ar@<.5ex>[d]\ar@<-.5ex>[d]\\
 f_0^*TM_2\ar[rr]&&f_0^*TM_2\oplus h_0^*TM_3
}
\end{equation}
is given by the canonical inclusions and the respective zero sections, and
\begin{equation}\label{normali}
\xymatrix{
\mathfrak{X}_1\rtimes G \ar@<.5ex>[d]\ar@<-.5ex>[d] \ar[rr]^-{N(i_2,i_1)}&&(M_1)^2\times 
(TM_2\ltimes g_0^*TM_3)\rtimes \mathfrak{X}_2\rtimes G\ar@<.5ex>[d]\ar@<-.5ex>[d]\\
 M_1\ar[rr]&& M_1\times g_0^*(TM_3)
}
\end{equation}
is given by the canonical source target maps, caonical inclusions, the respective zero sections and $f_1$ between $\mathfrak{X}_1$ and $\mathfrak{X}_2$. We obtain thus that these are two Lie groupoid immersions for which the respective normal bundle groupoids admit an explicit description again by proposition \ref{normalgrpdisoequivariant} and hence, in this case, an equivariant $K$-orientation and the associated Thom isomorphisms as explained above, see equation (\ref{KequivariantThom}). The fact that the groupoid equivariant Thom isomorphisms in $K$-theory satisfies the abstract Thom properties, either as in the general case of the Thom property \ref{ThomThom} or as its particular case used for concluding theorem \ref{pushforwardfunctII} proof, that is, equation (\ref{ThomThomnormalequation}), follows by classic Thom properties in topological $K$-theory for spaces, indeed, this is proved by Le Gall in \cite{LeGall} where he proves that the groupoid equivariant Thom isomorphism is contructed by a descent procedure in $KK$-theory from the classic Thom isomorphism. The Le Gall results as used in this paper are resumed in the appendix in \cite{CaWangAENS}, proposition A.3. The abstract axioms in this paper are of course a generalization of the properties recalled in ref.cit. We conclude that the point 1. above holds in the present proof case.

To conclude this proof we will justify that both $\dnc(i_2,i_1)$ and $\dnc(j_2,j_1)$ are K-oriented Lie groupoid immersions. As we will now see this is based on a general geometric principle and the functoriality of the deformation to the normal cone construction. Indeed, let us consider a commutative diagram of Lie groupoid immersions
\begin{equation}
\xymatrix{
G_1\ar[r]^-{i_1}\ar[d]_-{j_1}&G_2\ar[d]^-{j_2}\\
H_1\ar[r]_-{i_2}&H_2,
}
\end{equation}
Assume that all these immersions are $K$-oriented and admit Thom isomorphisms. Suppose further $N(j_2,j_1)$ and $N(i_2,i_1)$ are as well $K$-oriented Lie groupoids admitting Thom isomorphisms. All this is satisfied in our case, either by hypothesis or by the explicit computations above, (\ref{normalj}) and (\ref{normali}). 
In order to justify that $\dnc(j_2,j_1)$ is $K$-oriented (the same argument works for $\dnc(i_2,i_1)$), let us observe that there are groupoid cocycles, or groupoid generalized morphisms,
\begin{equation}
G_2 \dashrightarrow spin^c(n),
\end{equation} 
because $j_2$ is $K$-oriented (with $n=dim\, H_2-dim\, G_2$),
and 
\begin{equation}
N(G_2,G_1) \dashrightarrow spin^c(n),
\end{equation} 
because $N(j_2,j_1)$ is $K$-oriented (with the same $n=dim\, H_2-dim\, G_2$), and that these two cocycles fit exactly as single groupoid cocycle
\begin{equation}
\dnc(G_2,G_1) \dashrightarrow spin^c(n),
\end{equation}
by functoriality of the deformation to the normal cone construction since the first cocycle above is induced by $j_2$ and the second by the differential of $j_2$ in the normal direction and hence the induced cocycle above corresponds precisely to the morphism $\dnc(j_2,j_1)$, thus defining a compatible $spin^c$-strucure on the associated normal bundle, and in particular gives the fact that $\dnc(j_2,j_1)$ is K-oriented.
\end{proof}

\subsubsection{The twisted case}\label{subsectiontwistedorb}

As we explainded in the introduction, the proofs of the main theorems of this paper are based on the abstraction of the proof of the wrong way functoriality theorem for twisted differentiable stacks uing deformation groupoids, theorem 4.2 in \cite{CaWangAENS}. Hence, in the presence of twistings, the proof of the theorem above, when applied to the associated central extensions as in ref.cit, applies as well for the twisted case. We state the theorem for completeness.

Let $G\rightrightarrows G_0$ be a Lie groupoid with a twisting\footnote{Given by a generalized morphism $G\dashrightarrow PU(H$ or equivalently by an $S^1$-central extension.} $\alpha$, as in \cite{CaWangAENS}. Let 
$\mathfrak{X}\rightrightarrows M$ be a $G$-orbifold groupoid, then the canonical Lie groupoid morphism
\begin{equation}
\mathfrak{X}\rtimes G\longrightarrow G
\end{equation} 
induces on $\mathfrak{X}\rtimes G$ a twisting that we still denote by $\alpha$ to not load the notation. Let $\mathfrak{X}_i \rightrightarrows M_i$ ($i=1,2$) be two orbifold groupoids. Consider a $G$-equivariant Lie groupoid morphism
\begin{equation}
f:\mathfrak{X}_1\longrightarrow \mathfrak{X}_2.
\end{equation}
As in \cite{CaWangAENS} there is an orientation twisting on $\mathfrak{X}_1\rtimes G$ associated to the $\mathfrak{X}_1\rtimes G$ vector bundle $T^*M_1\oplus f_0^*(T^*M_2)$. Definition 4.1 in \cite{CaWangAENS} extends word by word to the orbifold groupoid case to give a pushforward morphism
\begin{equation}
\xymatrix{
K^*_G(\mathfrak{X}_ 1,\alpha+o_f)\ar[r]^-{f_!}&K^*_G(\mathfrak{X}_2,\alpha)
}
\end{equation}
where $K^*_G(\mathfrak{X},\beta)$ stands for the K-theory of the twisted $C^*$-algebra $K_*(C^*(\mathfrak{X}\rtimes G,\beta))$. The wrong way functoriality theorem states as follows:

\begin{theorem}[Wrong way functoriality for twisted equivariant orbifold K-theory]
Let $G$ be a Lie groupoid with a twisting $\alpha$.
Let 
\begin{equation}\label{orbifoldsquare}
\xymatrix{
\mathfrak{X}_ 1\ar[r]^-{f}\ar[dr]_-{h}&\mathfrak{X}_2\ar[d]^-{g}\\
&\mathfrak{X}_3
}
\end{equation}
be a commutative diagram $K$-oriented $G$-equivariant Lie groupoid morphisms between orbifold groupoids. Then, the following diagram of abelian group morphisms commutes:
\begin{equation}
\xymatrix{
K^*_G(\mathfrak{X}_ 1,\alpha+o_h)\ar[r]^-{f_!}\ar[dr]_-{h_!}&K^*_G(\mathfrak{X}_2,\alpha+o_g)\ar[d]^-{g_!}\\
&K^*_G(\mathfrak{X}_3,\alpha)
}
\end{equation}
\end{theorem}

\appendix

\section{Semi-direct product groupoids}

In this paper, several semi-direct product groupoids are used, in this section we explicit them.

\subsection{A groupoid acting on a groupoid}\label{subsectiongrpdactinggrpd}

Let $G\rightrightarrows M$ be a Lie groupoid. Remember that a (right, but also some left spaces will appear) $G$-space consists of a manifold $P$ together with a moment  map $\pi_P:P\to M$ (a smooth map) and partially defined action map (smooth in our case)
\begin{equation}
\theta:P\times_{\pi_P,t_G}G \to P
\end{equation} 
such that, denoting $\theta(\gamma,g):= \gamma\cdot g$, whenever $\pi_P(\gamma)=t_G(g)$, the following action axioms are satisfied:
\begin{enumerate}
\item $\pi_P(\gamma\cdot g)=s_G(g)$.
\item $(\gamma\cdot g)\cdot h=\gamma\cdot (gh)$ for $\gamma\in P$ and $g,h\in^ G$ whenever defined.
\item $\gamma\cdot x=\gamma$ for $\gamma\in P$ and $\pi_P(\gamma)=x\in M$ seen as a unit arrow in $G$.
\end{enumerate}

\begin{definition}[A groupoid acting on a groupoid]
Let $H\rightrightarrows P$ a Lie groupoid. An action of the Lie groupoid $G$ on the Lie groupoid $H$, or simply an action of $G$ on $H$, consists on $G$-actions on the spaces $H$ and $P$, with respective moments maps $\pi_H$ and $\pi_P$, such that 
\begin{enumerate}
\item All the structural groupoids maps of $H$ commute with the moment maps $\pi_H$ and $\pi_P$.
\item All the structural groupoids maps of $H$ are $G$-equivariant (in the obvious way, it makes sense since first required property above).
\end{enumerate}
\end{definition}

\begin{definition}[The semi-direct product groupoid associated to a groupoid acting on a groupoid]
Let $G\rightrightarrows M$ be a Lie groupoid. Suppose there is a Lie groupoid $H\rightrightarrows P$ and a Lie groupoid action of $G$ on the Lie groupoid $H$. The associated semi-direct product groupoid
\begin{equation}
H\rtimes G \rightrightarrows P
\end{equation}
is the Lie groupoid described as follows:
\begin{enumerate}
\item The arrow manifold is
\begin{equation}
H\rtimes G= H\times_{\pi_H,t_G}G.
\end{equation}
\item The source and target maps:
For $(\gamma,g)\in H\times_{\pi_H,t_G}G$
\begin{equation}
s(\gamma,g):=s_H(\gamma)\cdot g
\end{equation}
and 
\begin{equation}
t(\gamma,g):=t_H(\gamma).
\end{equation}
\item The multiplication:
\begin{equation}
(\gamma,g)\cdot (\eta,h):=(\gamma\cdot (\eta\cdot g^{-1}),gh)
\end{equation}
defined if $\pi_H(\gamma)=t_G(g)$, $\pi_H(\eta)=t_G(h)$ and $s_H(\gamma)\cdot g=t_H(\eta)$.
\item The unit map:
\begin{equation}
u(x):= (u_H(x),u_G(\pi_P(x)))
\end{equation}
for $x\in P$.
\item The inverse:
\begin{equation}
(\gamma,g)^{-1}:=(\gamma^{-1}\cdot g,g^{-1}).
\end{equation}
\end{enumerate}
\end{definition}

It is straightforward to verify that with the above described structural maps we obtain indeed a Lie groupoid
\begin{equation}
H\rtimes G \rightrightarrows P.
\end{equation}

General examples, whose particular cases appear in this paper, include:

\begin{enumerate}
\item A Lie group $G$ acting on a Lie groupoid $H\rightrightarrows P$. In this case the moment maps are all trivial and hence the only condition for having a groupoid action is to have group actions on $H$ and $P$ such that all the groupoid structural maps of $H$ are $G$-equivariant. In such case, the associated semi-direct product groupoid $H \rtimes G$ has underlying total space $H\times G$.
\item An action of a Lie groupoid $G\rightrightarrows M$ on a space $P$ can be seen as a Lie groupoid action on the identity Lie groupoid $P\rightrightarrows P$. The associated semi-direct product groupoid is the "classic" associated action groupoid
$P\rtimes G\rightrightarrows P.$
\end{enumerate}

\subsection{Explicit examples used in this paper}

{\bf The action groupoid of a vector bundle map:} Let $E,F$ be two vector bundles over a space $X$. Let $\theta:E\to F$ be a vector bundle morphism over the identity of $X$. There is an associated (left in this paper) action of the vector bundle Lie groupoid $E\rightrightarrows X$ on the unit Lie groupoid $F\rightrightarrows F$, the action is given by fiberwise translation by $\theta$: 
For $x\in X$, $V\in E_x$ and $W\in F_x$,
$$(x,V)\cdot (x,W):= (x,\theta(V)+W).$$
The associated semi-direct product groupoid
$E\ltimes F\rightrightarrows F$
is called the action groupoid (or semi-direct product groupoid) associated to $\theta$. 

An example of this situation appears several times in this paper. For example, for $f_0:M_1\to M_2$ a smooth map between two manifolds, there is an induced vector bundle morphism over $M_1$ given by the differential
$$df_0:TM_1\to f_0^*TM_2,$$
and hence the associated Lie groupoid
$$TM_1\ltimes f_0^*TM_2\rightrightarrows f_0^*TM_2.$$

{\bf The normal groupoid of an orbifold morphism}
Let $\mathfrak{X}_i \rightrightarrows M_i$ be two orbifold groupoids. Consider a Lie groupoid morphism
\begin{equation}
f:\mathfrak{X}_1\longrightarrow \mathfrak{X}_2.
\end{equation}

The terminology for the groupoid we present below is justified by proposition \ref{normalgrpdiso}. Before describing it we need to describe some actions:

\begin{itemize}
\item The Lie groupoid action of $\mathfrak{X}_1$ on the vector bundle Lie groupoid $TM_1\rightrightarrows M_1$. It is induced by the (right) action of $\mathfrak{X}_1$ on the manifold $TM_1$:

For $(x,V_1)\in T_xM_1$ and $g\in \mathfrak{X}_1$ with $t_{\mathfrak{X}_1}(g)=x$

$$(x,V_1)\cdot g:=(s_{\mathfrak{X}_1}(g),(d_gs_{\mathfrak{X}_1}\circ (d_gt_{\mathfrak{X}_1})^{-1})(V_1))\in T_{s_{\mathfrak{X}_1}(g)}M_1,$$

where we have used that $\mathfrak{X}_1$ is \'etale so that $(d_gt_{\mathfrak{X}_1})^{-1}$ makes sense. 

\item The Lie groupoid action of $\mathfrak{X}_1$ on the unit Lie groupoid $f_0^*TM_2\rightrightarrows f_0^*TM_2$. It is induced by the Lie groupoid morphism $f$ and the (right) action of $\mathfrak{X}_2$ on the manifold $f_0^*TM_2$ given as the above case:

For $(x,V_2)\in (f_0^*TM_2)_{x}$ (i.e. $V_2\in T_{f_0(x)}M_2$) and $g\in \mathfrak{X}_1$ with $t_{\mathfrak{X}_1}(g)=x$

$$(x,V_2)\cdot g:=(s_{\mathfrak{X}_1}(g),(d_{f(g)}s_{\mathfrak{X}_2}\circ (d_{f(g)}t_{\mathfrak{X}_2})^{-1})(V_2))
\in (f_0^*TM_2)_{s_{\mathfrak{X}_1}(g)},$$
where the notations are as above.
\end{itemize}

The two actions above combine to give a diagonal Lie groupoid action of $\mathfrak{X}_1$ on the Lie groupoid 
$$TM_1\ltimes f_0^*TM_2\rightrightarrows f_0^*TM_2$$
and give a semi-direct product groupoid
\begin{equation}\label{normalgrpdorbimorphism}
(TM_1\ltimes f_0^*TM_2)\rtimes \mathfrak{X}_1\rightrightarrows f_0^*TM_2
\end{equation}
that we call {\it the normal groupoid of the orbifold morphism $f$}.

\vspace{2mm}

{\bf The Fourier normal groupoid of an orbifold morphism}

There is no Fourier transform or related for Lie groupoids, however the Lie groupoid described below has the property that its $C^*$-algebra is isomorphic the $C^*$-
algebra of the previous example, the isomorphism given a fiberwise Fourier transform, hence the terminology.

For describing these kind of groupoids, let us consider $\mathfrak{X}_i \rightrightarrows M_i$ two orbifold groupoids, for $i= 1,2$, together with a Lie groupoid morphism
\begin{equation}
f:\mathfrak{X}_1\longrightarrow \mathfrak{X}_2.
\end{equation}

There is an action of the Lie groupoid $\mathfrak{X}_1$ on the Lie groupoid $T^*M_1\rightrightarrows T^*M_1$ given as follows:

For $(x,\xi_1)\in T_x^*M_1$ and $g\in \mathfrak{X}_1$ with $t_{\mathfrak{X}_1}(g)=x$

\begin{equation}
(x,\xi_1)\cdot g:= (s_{\mathfrak{X}_1}(g),\xi_1 \circ d_{g}t_{\mathfrak{X}_1}\circ (d_{g}s_{\mathfrak{X}_1})^{-1})\in T_{s_{\mathfrak{X}_1}(g)}^*M_1
\end{equation}
with notation as above.

With the above action in hand we obtain a Lie groupoid action of $\mathfrak{X}_1$ on the Lie groupoid $T^*M_1\oplus f_0^*TM_2\rightrightarrows T^*M_1\oplus f_0^*TM_2$ giving the semi-direct product groupoid 
\begin{equation}
(T^*M_1\oplus f_0^*TM_2)\rtimes \mathfrak{X}_1\rightrightarrows T^*M_1\oplus f_0^*TM_2
\end{equation}
that we call {\it the Fourier normal groupoid of the orbifold morphism $f$}.

\section{Equivariant K-theory Thom isomorphism for orbifold groupoids}\label{appendixThom}

Let $\mathfrak{X}_i \rightrightarrows M_i$ be two orbifold groupoids, for $i=1,2$. Consider a Lie groupoid morphism
\begin{equation}
f:\mathfrak{X}_1\longrightarrow \mathfrak{X}_2.
\end{equation}
Assume there is a Lie groupoid $G$ such that $f$ is $G$-equivariant and $G$-$K^*$-oriented.

In this appendix we will describe the Thom isomorphism
\begin{equation}\label{appendixThomf}
T_{\tilde{f}}:K^*_G(\mathfrak{X}_1)\longrightarrow K^*_G(N(\tilde{f})) 
\end{equation}
where $N(\tilde{f})$ stands for the normal groupoid of the groupoid immersion $$(f, (r,s)): \mathfrak{X}_1 \longrightarrow  \mathfrak{X}_2 \times (M_1 \times M_1).
$$

In order to illustrate its construction, let us consider the following simpler situation first. Let $E_1, E_2$ be two vector bundles over a manifold $M$, and let $\theta: E_1 \to E_2$ be a vector bundle map over the identity of $M$. Supppose the vector bundle $E_1^*\oplus E_2\to M$ admits a $spin^c$-structure. We will define an isomorphism
\begin{equation}
K^*(M)\xrightarrow{Th_\theta} K^*(E_1\ltimes E_2)
\end{equation}
where $E_1\ltimes E_2\rightrightarrows E_2$ stands for the action groupoid associated to the canonical action induced from the vector bundle morphism $\theta$. It will be given by the composition of the following three isomorphisms:

\begin{enumerate}
\item The ``classical'' Thom isomorphism
\begin{equation}
\xymatrix{
K^*(M)\ar[rr]^-{Thom_{E_1^*\oplus E_2}}&& K^*(E_1^*\oplus E_2)
}
\end{equation}
in topological K-theory of spaces.
\item The isomorphism
\begin{equation}
\xymatrix{
K^*(E_1^*\oplus E_2\rightrightarrows E_1^*\oplus E_2) \ar[r]^-{\mathcal{F}} & K^*(E_1\oplus E_2\rightrightarrows E_2)
}
\end{equation}
induced from the isomorphism of $C^*$-algebras
\begin{equation}
C_0(E_1^*\oplus E_2\rightrightarrows E_1^*\oplus E_2)=C^*(E_1^*\oplus E_2\rightrightarrows E_1^*\oplus E_2)\cong C^*(E_1\oplus E_2\rightrightarrows E_2)
\end{equation}
induced by the Fourier isomorphism in the $E_1$-direction.
\item The isomorphism
\begin{equation}\label{deflineariso}
K^*(E_1\oplus E_2\rightrightarrows E_2)\cong K^*(E_1\ltimes E_2\rightrightarrows E_2)
\end{equation}
induced from the deformation of the action by $\theta$ that we now briefly recall:

Consider the groupoid action of the vector bundle groupoid $E_1\times [0,1]\rightrightarrows M\times [0,1]$ on the space groupoid $E_2\times [0,1]\rightrightarrows E_2\times [0,1]$ (with obvious momentum maps) given by
$$(v_1,\epsilon)\cdot (v_2,\epsilon):= (\epsilon\cdot \theta(v_1)+v_2,\epsilon)$$
for $v_i\in E_i$ in the same fiber over $M$. This induce an action Lie groupoid
\begin{equation}
\mathcal{H}=(E_1\times [0,1])\ltimes (E_2\times [0,1])\rightrightarrows E_2\times [0,1]
\end{equation}
When restricted to $\epsilon=0$ this groupoid gives the groupoid $E_1\oplus E_2\rightrightarrows E_2$, and when restricted to $\epsilon=1$ one has the groupoid $E_1\ltimes E_2\rightrightarrows E_2$.
It can be shown that in this case both evaluation morphisms
\begin{equation}
C^*(\mathcal{H})\stackrel{e_0}{\longrightarrow} C^*(E_1\oplus E_2)
\end{equation}
and 
\begin{equation}
C^*(\mathcal{H})\stackrel{e_1}{\longrightarrow} C^*(E_1\ltimes E_2)
\end{equation}
have contractible kernels and hence induce isomorphisms in $K$-theory. The isomorphism (\ref{deflineariso}) above is given precisely by the induced morphism from the evaluations above
\begin{equation}
(e_1)_ *\circ (e_0)_ *^{-1}:K^*(E_1\oplus E_2\rightrightarrows E_2)\to K^*(E_1\ltimes E_2\rightrightarrows E_2).
\end{equation}
\end{enumerate}
The definition for the general case follows the same lines applied to the case $E_1=TM_1$ and $E_2=f_0^*TM_2$ with the extra data of the groupoid action by $\mathfrak{X}_1\rtimes G$. We explain it in the following definition.
%
%
\begin{definition}\label{appendixThomfdef}
Let $\mathfrak{X}_i \rightrightarrows M_i$ be two orbifold groupoids. Consider a Lie groupoid morphism
\begin{equation}
f:\mathfrak{X}_1\longrightarrow \mathfrak{X}_2.
\end{equation}
Assume there is a Lie groupoid $G$ such that $f$ is $G$-equivariant and $G$-$K$-oriented. The Thom isomorphism
\begin{equation}
T_{\tilde{f}}:K^*_G(\mathfrak{X}_1)\to K^*_G(N(\tilde{f})) 
\end{equation}
is given as the composition of the following four isomorphisms:
\begin{enumerate}
\item The $\mathfrak{X}_1\rtimes G$-equivariant Thom isomorphism
\begin{equation}
\xymatrix{
K^*_G(\mathfrak{X}_1)=K^*(M_1\rtimes (\mathfrak{X}_1\rtimes G))\ar[rr]^-{Thom}&& 
K^*((T^*M_1\oplus f_0^*TM_2)\rtimes (\mathfrak{X}_1\rtimes G))
}
\end{equation}
associated to the $\mathfrak{X}_1\rtimes G$-$spin^c$ vector bundle $T^*M_1\oplus f_0^*TM_2\to M_1$.
\item The isomorphism
\begin{equation}
\xymatrix{
K^*((T^*M_1^*\oplus f_0^*TM_2)\rtimes (\mathfrak{X}_1\rtimes G)) \ar[r]^-{\mathcal{F}} & K^*((TM_1\oplus f_0^*TM_2)\rtimes (\mathfrak{X}_1\rtimes G))
}
\end{equation}
induced from the isomorphism of $C^*$-algebras
\begin{equation}
C^*((T^*M_1\oplus f_0^*TM_2)\rtimes (\mathfrak{X}_1\rtimes G))\cong C^*((T^*M_1\oplus f_0^*TM_2)\rtimes (\mathfrak{X}_1\rtimes G))
\end{equation}
induced by the Fourier isomorphism in the $TM_1$-direction.
\item The isomorphism
\begin{equation}\label{deflineariso}
K^*((TM_1\oplus f_0^*TM_2)\rtimes (\mathfrak{X}_1\rtimes G))\cong K^*((TM_1\ltimes f_0^*TM_2)\rtimes (\mathfrak{X}_1\rtimes G))
\end{equation}
induced from the deformation of the action by $df_0$.
\item The isomorphism 
\begin{equation}
K^*((TM_1\ltimes f_0^*TM_2)\rtimes (\mathfrak{X}_1\rtimes G))\cong K^*(N(\tilde{f})\rtimes G)=K^*_G(N(\tilde{f}))
\end{equation}
induced from the groupoid isomorphism \ref{normalgrpdisoequivariant}.
\end{enumerate}
\end{definition}

\bibliographystyle{plain}
\bibliography{bibliographie}

\end{document}